\documentclass[reqno,12pt]{amsart}
\usepackage{amssymb,amsfonts,graphicx}
\usepackage{geometry}

\usepackage{mathrsfs}
\usepackage{float}

\usepackage[toc,page,title,titletoc,header]{appendix}
\usepackage{epsfig}
\usepackage{mathrsfs}
\usepackage{float}

\usepackage{amssymb}
\usepackage{mathrsfs}
\usepackage{float}
\usepackage{amsmath}
\usepackage{amsfonts,amssymb}
\usepackage{graphicx,color}
\usepackage{bm}

\usepackage{fancyhdr}
\allowdisplaybreaks

\begin{document}
\setlength{\parindent}{1.5em}
\theoremstyle{plain}
\newtheorem{theorem}{Theorem}[section]
\newtheorem{definition}{Definition}[section]
\newtheorem{lemma}{Lemma}[section]
\newtheorem{Example}{Example}[section]
\newtheorem{cor}{Corollary}[section]
\newtheorem{remark}{Remark}[section]
\newtheorem{property}{Property}[section]
\newtheorem{proposition}{Proposition}[section]


\title{Formation of singularities in one-dimensional Chaplygin gas}
\maketitle
 \centerline{De-Xing Kong$^{a} $\quad Changhua Wei$^{b,*}$ \quad Qiang Zhang $^{c}$}
\begin{center}
$^*\!${\it {\small Corresponding author }}\\
$^a\!${\it {\small  Department of Mathematics, Zhejiang University, Hangzhou 310027, China}}\\
{\it  {\small E-mail address: dkong@zju.edu.cn}}\\
$^{b}\!${\it {\small Department of Mathematics, Zhejiang University,
Hangzhou 310027, China}}\\
{\it  {\small E-mail address: wch\_19861125@163.com}}\\
$^c\!${\it {\small  Department of Mathematics, City University of Hong Kong, China}}\\
{\it  {\small E-mail address: mazq@cityu.edu.hk}}\\

\end{center}

        \begin{abstract}
        In this paper we investigate the formation and propagation of singularities for the system for one-dimensional Chaplygin
        gas. In particular, under suitable assumptions we construct a physical solution with a new type of singularities
        called ``Delta-like" solution for this kind of quasilinear hyperbolic system with linearly degenerate characteristics.
        By careful analysis, we study the behavior of the solution in a neighborhood of a blowup point. The formation of
        this new kind of singularities is due to the envelope of the different families of characteristics instead of the
        same family of characteristics in the traditional situation. This shows that the blowup phenomenon of solution
        for the system with linearly degenerate characteristics is quite different from the situation of shock formation
        for the system with genuinely nonlinear characteristics. Different initial data can
        lead to kinds of different Delta-like singularities: the Delta-like singularity with point-shape and Delta-like singularity with line-shape.

        \vskip 6mm
        \noindent{\bf Key words and phrases:}  System for Chaplygin gas, linearly degenerate characteristic,
        blowup, singularity, Delta-like singularity.

        \vskip 3mm

        \noindent{\bf 2000 Mathematics Subject Classification}: 35L45, 35L67, 76N15.
        \end{abstract}

\baselineskip=7mm

  \section{Introduction}
As we know, smooth solutions of nonlinear hyperbolic
systems generally exist in finite time even if initial data is
sufficiently smooth and small. After this time, only weak solutions
can be defined. Therefore, the following questions arise naturally:\\
(I) When and where do the solutions blow up?\\
(II) What quantities blow up? how do they blow up?\\
(III) What kinds of singularities appear? How do the
singularities propagate?

\noindent  These questions are very important in mathematics and
physics. For questions (I) and (II), some methods were established
and many results were obtained (see \cite{6}, \cite{9}, \cite{13},
\cite{14}). As for question (III), since this kind of nonlinear
phenomena is too complicated, up to now, only a few results on
shock formation are known. For a single conservation law, these
questions can be solved well by the usual characteristic method (see
\cite{13}). For the $p$-system, Lebaud \cite{12} investigates the
problem of shock formation from the simple waves, namely, under the
hypothesis that one Riemann invariant keeps constant. Kong \cite{7} studies the formation and propagation of
singularities (in particular, the shock formation) for $2\times 2$
quasilinear hyperbolic systems with genuinely nonlinear
characteristics. For more complete introduction of the blowup for nonlinear hyperbolic equations, one can refer to Alinhac \cite{A}. Recently, Christodoulou \cite{Ch1} considers the relativistic Euler equations for a
perfect fluid with an arbitrary equation of state and under certain smallness assumptions on the
size of initial data, he obtains a remarkable and complete picture of the formation of shock waves in three dimensions
(one can also refer to Christodoulou and Miao \cite{Ch2}).

Hyperbolic systems with linearly degenerate characteristics play an
important role in mathematics and physics. For example, many
important equations arising from geometry and physics can be reduced to
this class of PDEs. The typical examples include the equations for
extremal time-like surfaces in Minkowski space, the Born-Infeld
equation in nonlinear theory of the electromagnetic field, the
system for Chaplygin gas (see \cite{1}, \cite{10} and
\cite{11}). However, up to now, for hyperbolic systems with linearly
degenerate characteristics, most of results are on the global
existence of solutions. Only a few results on formation of
singularities are known. Recently, Eggers and Hoppe \cite{3}
investigates the Born-Infeld equation, derives self-similar string solutions in a graph representation near the point of singularity formation and investigates the formation of a swallowtail
singularity. In this paper, we study the formation and propagation of
singularities in one-dimensional Chaplygin gas. In particular,
we investigate the cusp-type singularities of solutions. By the same method, similar
results hold for the equations for extremal time-like surfaces in
Minkowski space and for the Born-Infeld equation.

In this paper, we present a systematic analysis of the formation of cusp-type singularities (see \cite{A}) arising from some special smooth initial data; in particular, we provide a complete description of the solution close to the blowup point. Furthermore, based on this, we introduce the concept of ``Delta-like'' solution and construct several ``Delta-like'' solutions with applications in practice. In fact, the ``Delta-like'' solution is a weak solution which satisfies the definition of weak solution in classical sense (see \cite{14}), however, at the blowup point, the density function of the gas is infinite, this phenomenon is due to the concentration of mass of the gas in finite interval.

The rest of the paper is organized as follows. In Section 2, we give some
preliminaries on quasilinear hyperbolic system with linearly degenerate
characteristics. In Section 3, using the method of characteristic coordinates and
the singularity theory of smooth mappings, we give a
detailed analysis on the formation of singularities. In Section 4,
we present a complete description of the solution in the
neighborhood of the blowup point. In Section 5, we construct a
physical solution containing a new kind of singularity called
``Delta-like" singularity after the blowup time. In Section 6, under different
assumptions on initial data, we construct several different weak
solutions with ``Delta-like'' singularities which are named the
``Delta-like singularity with point-shape'' and ``Delta-like singularity with line-shape'', respectively. Section 7 is for conclusion.

\section{Preliminaries}
In this section, we consider one-dimensional system of isentropic gas in Eulerian representation
$$
\left\{\begin{array}{l}
{\displaystyle \partial_t\rho+\partial_x(\rho u)=0 },\vspace{2mm}\\
{\displaystyle \partial_t (\rho u)+\partial_x(\rho u^2+p)=0}, \vspace{2mm}\\
\end{array}\right.
\eqno(2.1)
$$
where $t\in\mathbb R^{+}$ and $x\in\mathbb R$ stand for the time variable and spatial variable, respectively, while $ \rho=\rho(t, x)$ and
  $u=u(t, x)$ denote the density and the velocity, respectively, and
$ p(t, x)$ is the pressure which is a function of $ \rho$ given by
{$$ p=p_0-\frac{\mu^2}{\rho}.\eqno(2.2) $$\vspace{2mm}
Here $ p_0 $ and $ \mu $ are two positive constants. The system (2.1) with (2.2) describes the motion of
a perfect fluid characterized by the pressure-density relation (known as the Chaplygin or
von K\'{a}rm\'{a}n-Tsien pressure law). This endows the system a highly symmetric structure. This is evident if we adopt the local sound speed
$ c=c(t,x)=[p'(\rho)]^{1/2}$ and the usual mean velocity of the fluid $ u $ as dependent variables. In this case, the system reads
$$
\partial_{t}U+A(U)\partial_{x}U=0,
\eqno(2.3)
$$
where
\begin{displaymath}
U=\left(\begin{array}{c}c\\u
\end{array}
\right)\quad \rm{and}\quad
A(U)=\left(\begin{array}{cc}u&-c\\-c&u
\end{array}
\right).
\end{displaymath}
Obviously, the eigenvalues of $A(U)$ read
$$
 \lambda_{-}=u-c,\quad\lambda_{+}=u+c.
 \eqno(2.4)
$$
Moreover, it is easy to verify that $\lambda_{\pm}=\lambda_{\pm}(t,x)$ are Riemann invariants.
Under the Riemann invariants, the system (2.1) can be reduced to
 $$\left\{\begin{array}{l}
{\displaystyle \partial_t\lambda_{-}+\lambda_{+}\partial_x\lambda_{-}=0 },\vspace{2mm}\\
{\displaystyle \partial_t\lambda_{+}+\lambda_{-}\partial_x\lambda_{+}=0}. \vspace{2mm}\\
\end{array}\right.
\eqno(2.5)
$$

Consider the Cauchy problem for the system (2.1) with the following initial data
$$
t=0: \rho=\rho_{0}(x),\quad u=u_{0}(x),
\eqno(2.6)
$$
where $ \rho_{0}(x)$  and $ u_{0}(x)$  are two suitably smooth functions with bounded $ C^{2} $ norm.
For consistency, let
$$
\lambda_{\pm}(0,x)=\Lambda_{\pm}(x)\triangleq u_{0}(x)\pm\frac{\mu}{\rho_{0}(x)}.
\eqno(2.7)
$$
Thus, studying the system (2.1) with initial data (2.6) is equivalent to studying the system (2.5) with initial data (2.7) in the existence domain of classical solutions.

In the existence domain of the classical solution of (2.5), (2.7), we recall the definition of characteristics and denote two characteristics starting from $(0,\alpha)$ by
$$
x=x^{+}(t,\alpha),\quad x=x^{-}(t,\alpha),
$$
respectively, which satisfy
$$\left\{\begin{array}{l}
{\displaystyle \frac{dx^{+}(t,\alpha)}{dt}=\lambda_{-}(t,x^{+}(t,\alpha)) },\vspace{2mm}\\
{\displaystyle t=0:x^{+}(0,\alpha)=\alpha} \vspace{2mm}\\
\end{array}\right.
\eqno(2.8)
$$
and
$$\left\{\begin{array}{l}
{\displaystyle \frac{dx^{-}(t,\alpha)}{dt}=\lambda_{+}(t,x^{-}(t,\alpha))},\vspace{2mm}\\
{\displaystyle t=0:x^{-}(0,\alpha)=\alpha}, \vspace{2mm}\\
\end{array}\right.
\eqno(2.9)
$$ respectively.
Let $(\alpha, \beta)$ be the characteristic parameters defined as follows.
For any $(t,x)$ in the maximal domain of definition of a smooth solution , we define $\beta(t,x)$ by $\beta(t,x)=x^{+}(0;t,x)$ where $x^{+}(0;t,x)$ is the unique solution of the ODE
$$
\frac{df(s;t,x)}{ds}=\lambda_{-}(s,f(s;t,x)),
$$
with initial condition $f(t)=x$. $\alpha(t,x)$ is defined similarly. The geometric meaning of $\alpha(t,x)$ and $\beta(t,x)$ are shown in Figure 1.
\begin{lemma}By (2.5),
$\lambda_{+}(t,x)$ is constant along the curve $x=x^{+}(t,\alpha)$, while  $\lambda_{-}(t,x)$ is constant along the curve $x=x^{-}(t,\alpha)$.
\end{lemma}
The following lemma can be found in Kong-Zhang \cite{11}
\begin{lemma}
In terms of characteristic parameters $(\alpha,\beta)$ introduced above, it holds that
$$
t(\alpha,\beta)=\int_\alpha^\beta\frac{1}{\Lambda_{+}(\zeta)-\Lambda_{-}(\zeta)}d\zeta,
\eqno(2.10)
$$
$$
x(\alpha,\beta)=\frac{1}{2}\left\{{\alpha+\beta+\int_\alpha^\beta\frac{\Lambda_{+}(\zeta)+\Lambda_{-}(\zeta)}{\Lambda_{+}(\zeta)-\Lambda_{-}(\zeta)}}d\zeta\right\}.
\eqno(2.11)
$$
\end{lemma}

\begin{figure}[h]
\centering
\includegraphics[]{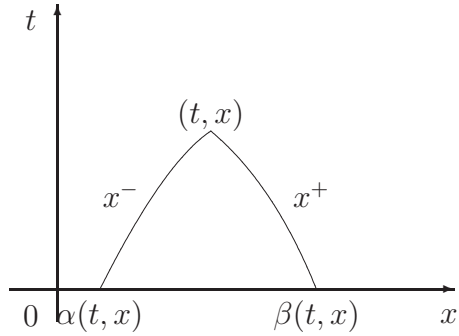}
\caption{The geometric meaning of characteristic coordinates $\alpha$ and $\beta$.}
\end{figure}
 \section{ formation of singularity}
 This section is devoted to the formation of the cusp-type singularity under suitable assumptions on initial data. The following lemma, which can be found in Kong \cite{10}, plays an important role in our discussion.
 \begin{lemma}
  Adopt the notations in Section 2. If there exists $\alpha$ such that $\Lambda_{-}(\alpha)\neq\lambda_{+}(t,x^{-}(t,\alpha))$ for $t\geq0$, then it holds that
 $$
 \frac{\partial x^{-}(t,\alpha)}{\partial\alpha}=\frac{\lambda_{+}(t,x^{-}(t,\alpha))-\Lambda_{-}(\alpha)}{\Lambda_{+}(\alpha)-\Lambda_{-}(\alpha)}
 \eqno(3.1)
 $$and
 $$
 \frac{\partial\lambda_{-}(t,x)}{\partial x}\bigg|_{x=x^{-}(t,\alpha)}=\Lambda^{'}_{-}(\alpha)\frac{\Lambda_{+}(\alpha)-\Lambda_{-}(\alpha)}{\lambda_{+}(t,x^{-}(t,\alpha))-\Lambda_{-}(\alpha)}.
 \eqno(3.2)
 $$
 Similar result holds for $x=x^{+}(t,\beta)$ and $\lambda_{+}(t,x^{+}(t,\beta))$.
 \end{lemma}
\begin{remark}
It follows from $(3.2)$ that if there exists time $t_{0}>0$ which satisfies  $$\lambda_{+}(t_{0},x^{-}(t_{0},\alpha))=\Lambda_{-}(\alpha),\:\: \Lambda_{-}(\alpha)\neq\Lambda_{+}(\alpha)\:\: \text{and} \:\:\Lambda^{'}_{-}(\alpha)\neq0,\;\text{for some}\, \alpha,$$then the solution of the Cauchy problem (2.5), (2.7) must blow up at the time $t_{0}$. By the theory of characteristic method, we observe that $\lambda_{-}(t,x)$ and $\lambda_{+}(t,x)$ are bounded, while $(\lambda_{-})_{x}$ and $(\lambda_{+})_{x}$ tend to the infinity as $t$ goes to $t_{0}$.

 It is well known that the formation of traditional blowup, e.g., the formation of ``shock wave'' is due to the envelope of the same family of characteristics (see \cite{A,9}). However, in this paper, we shall investigate a new phenomenon on the formation of singularities which is based on the envelope of different families of characteristics (see Figure 2).
\end{remark}
\begin{figure}[h]
\centering
\includegraphics[]{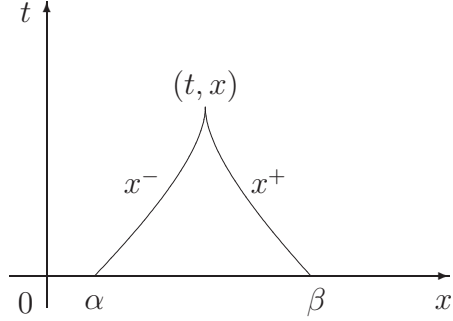}
\caption{The envelope of different families of characteristics}
\end{figure}
To do so, we suppose that the initial data $\Lambda_{-}(x)$ and $\Lambda_{+}(x)$ are suitably smooth
 functions and satisfy the following assumptions:

Assumption (H1):
$$
\Lambda_{-}(x)<\Lambda_{+}(x),\quad \forall\:  x\in \mathbb{R}.
\eqno(3.3)
$$

Assumption (H2):
$$
\Lambda^{'}_{-}(x)<0\quad \text{and}\quad \Lambda^{'}_{+}(x)<0,\qquad\forall\:x\in \mathbb{R}.
\eqno(3.4)
$$
Define
 $$\Sigma=\{(\alpha,\beta)|\alpha<\beta\;\text{and}\;\Lambda_{-}(\alpha)=\Lambda_{+}(\beta)\}.
 \eqno(3.5)$$
  In order to avoid confusion, here and hereafter, we denote the variable of $\Lambda_{-}(x)$ by $\alpha$ and the variable of $\Lambda_{+}(x)$ by $\beta$.

 By (3.4) and (3.5), for $\forall\,(\alpha,\beta)\in\Sigma$, it holds that $\beta(\alpha)=\Lambda_{+}^{-1}\Lambda_{-}(\alpha)$. Define
 $$
 f(\alpha)\triangleq\frac{\Lambda^{'}_{-}(\alpha)}{\Lambda_{+}(\beta(\alpha))-\Lambda_{-}(\beta(\alpha))}
-\frac{\Lambda^{'}_{+}(\beta(\alpha))}{\Lambda_{+}(\alpha)-\Lambda_{-}(\alpha)},
\eqno(3.6)
 $$
 where $$\Lambda^{'}_{+}(\beta(\alpha))=\frac{d\Lambda_{+}(\beta)}{d\beta}\bigg|_{\beta=\beta(\alpha)}.$$
 We furthermore assume that there exists $(\alpha_{0},\beta_{0})\in\Sigma$ such that

 Assumption (H3):
$$\Lambda_{-}(\alpha_{0})=\Lambda_{+}(\beta_{0}).
\eqno(3.7)
$$

Assumption (H4):
$$
f(\alpha_{0})=0.
\eqno(3.8)
$$

Assumption (H5):
$$
f^{'}(\alpha_{0})<0.
\eqno(3.9)
$$

For simplicity, without loss of generality, we may suppose that
$$
\Lambda_{-}(\alpha_{0})=\Lambda_{+}(\beta_{0})=0.
\eqno(3.10)
$$
This can be achieved by making a simple translation transform.
\begin{lemma}
Initial data set $\{(\Lambda_{+}(x),\Lambda_{-}(x))\}$ satisfying assumptions (H1)-(H5) is not empty.
\end{lemma}
\begin{proof}
We prove the lemma by construction.

Firstly, choose $\Lambda_{-}(x)$ such that it satisfies (3.4) and (3.10). Then, at the point $\alpha_{0}$ it holds that
$$
\Lambda_{-}(\alpha_{0})=0\quad\rm{and}\quad \Lambda^{'}_{-}(\alpha_{0})<0.
$$

Secondly, fix $\Lambda_{-}(x)$ and choose $\Lambda_{+}(x)>\Lambda_{-}(x)$ for all $x\in\mathbb R$. Moreover, at the point $\beta_{0}$ it satisfies (3.8) and $\Lambda_{+}(\beta_{0})=0$.

By (3.8), $\Lambda_{+}(x)$ satisfies
$$
\frac{\Lambda^{'}_{-}(\alpha_{0})}{\Lambda_{+}(\beta_{0})-\Lambda_{-}(\beta_{0})}=\frac{\Lambda^{'}_{+}(\beta_{0})}
{\Lambda_{+}(\alpha_{0})-\Lambda_{-}(\alpha_{0})},
$$
By assumption (H1), it holds that
$$
\Lambda_{+}(\beta_{0})-\Lambda_{-}(\beta_{0})>0\quad \rm{and}\quad \Lambda_{+}(\alpha_{0})-\Lambda_{-}(\alpha_{0})>0,
$$
and then
$$
\Lambda_{+}^{'}(\beta_{0})<0.
$$
Thus, we can choose $\Lambda_{+}(x)$ satisfing assumptions (H2)-(H4).

Finally, we prove that $\Lambda_{+}(x)$, constructed in the way mentioned above, satisfies the assumption (H5) for fixed $\Lambda_{-}(x)$.

In fact, by (3.8), if we fix the value of $\Lambda_{+}^{'}(\beta_{0})<0$, then $\Lambda_{+}(\alpha_{0})$ satisfies
$$
\Lambda_{+}(\alpha_{0})=-\frac{\Lambda^{'}_{+}(\beta_{0})\Lambda_{-}(\beta_{0})}{\Lambda^{'}_{-}(\alpha_{0})}>0.
$$
By (3.5) and (3.9), it must hold that
\begin{align*}
\begin{split}
f^{'}(\alpha_{0})&=-\frac{\Lambda^{''}_{-}(\alpha_{0})\Lambda_{-}(\beta_{0})+\Lambda^{'}_{-}(\alpha_{0})\left(\Lambda^{'}_{+}(\beta_{0})-
\Lambda^{'}_{-}(\beta_{0})\right)\frac{\Lambda^{'}_{-}(\alpha_{0})}{\Lambda^{'}_{+}(\beta_{0})}}{\Lambda^{2}_{-}(\beta_{0})}\\
&\quad -\frac{\Lambda^{''}_{+}(\beta_{0})\frac{\Lambda^{'}_{-}(\alpha_{0})}{\Lambda^{'}_{+}(\beta_{0})}\Lambda_{+}(\alpha_{0})
-\Lambda^{'}_{+}(\beta_{0})\left(\Lambda^{'}_{+}(\alpha_{0})-\Lambda^{'}_{-}(\alpha_{0})\right)}{\Lambda^{2}_{+}(\alpha_{0})}\\
&=\frac{\Lambda^{''}_{-}(\alpha_{0})\Lambda^{2}_{+}(\alpha_{0})-\Lambda^{2}_{-}(\beta_{0})\Lambda^{''}_{+}(\beta_{0})}
{-\Lambda_{-}(\beta_{0})\Lambda^{2}_{+}(\alpha_{0})} +\frac{
\Lambda^{'}_{+}(\beta_{0})\Lambda_{-}(\beta_{0})\left[\Lambda^{'}_{+}(\beta_{0})-\Lambda^{'}_{-}(\beta_{0})-\left(\Lambda^{'}_{+}(\alpha_{0})-\Lambda^{'}_{-}(\alpha_{0})\right)\right]}
{-\Lambda_{-}(\beta_{0})\Lambda^{2}_{+}(\alpha_{0})}\\
&<0,
\end{split}\tag{3.11}
\end{align*}
where $\beta_{0}=\beta(\alpha_{0})$. Since $\Lambda_{-}^{'}(x)<0$, by (3.4) and (3.5), we have
$$
\Lambda_{-}(\beta_{0})<0.
$$
Then, at $\beta_{0}=\beta(\alpha_{0})$, $\Lambda_{+}(x)$ should satisfy at $\alpha_{0}$ and $\beta_{0}$
$$
\Lambda^{2}_{-}(\beta_{0})\Lambda^{''}_{+}(\beta_{0})>\Lambda^{''}_{-}(\alpha_{0})\Lambda^{2}_{+}(\alpha_{0})+
\Lambda^{'}_{+}(\beta_{0})\Lambda_{-}(\beta_{0})\left[\Lambda^{'}_{+}(\beta_{0})-\Lambda^{'}_{-}(\beta_{0})-\left(\Lambda^{'}_{+}(\alpha_{0})-\Lambda^{'}_{-}(\alpha_{0})\right)\right].
$$

Obviously, there exists such a $\Lambda_{+}(x)$ such that the above inequality holds at the point $\alpha_{0}$ and $\beta_{0}=\beta(\alpha_{0})$. Therefore, it is easy to construct a smooth curve $\Lambda_{+}(x)$ to satisfy assumptions (H1)-(H5) once the information has been known at the points $\alpha_{0}$ and $\beta_{0}$.
\end{proof}
\begin{remark}
Assumptions (H3)-(H5) are restrictions to the initial data $\Lambda_{\pm}(x)$ at the points $\alpha_{0}$ and $\beta_{0}$, so we can change the shape of the curves $\Lambda_{\pm}(x)$ to make sure that they satisfy assumptions (H1)-(H5) once their properties at $P_{1}=(\alpha_{0},0)$ and $P_{2}=(\beta_{0},0)$ have been known. In fact, the geometric meaning of assumption (H4) is $$|BA|=|CD|$$
 as shown in Figure 3, where $|\cdot|$ denotes the distance in Euclidean space, $P_{1}B$ (resp. $P_{2}D$) stands for the tangential line of the curve $\Lambda_{-}(x)$ (resp. $\Lambda_{+}(x)$) at the point $P_{1}$ (resp. $P_{2}$).
\end{remark}
\begin{figure}[h]
\centering
\includegraphics[]{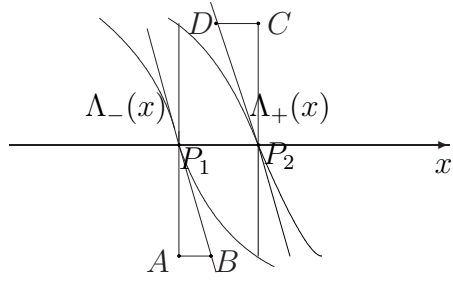}
\caption{ The geometric meaning of the assumption (H4). }
\end{figure}

By the existence and uniqueness theorem of a $C^{1}$ solution of the Cauchy problem for a quasilinear hyperbolic systems (see \cite{5}), under the assumptions (H1)-(H5), the Cauchy problem (2.5)-(2.7) has a unique $C^{1}$ solution $(\lambda_{-}(t,x),\lambda_{+}(t,x))$ in the domain $D(t_{0})\triangleq\{(t,x)|0\leq t<t_{0},-\infty<x<\infty\}$, where $t_{0}$ is just the blowup time, i.e., the life span of the $C^{1}$ solution of the Cauchy problem (2.5)-(2.7). Throughout the paper, we refer $D(t_{0})$ as the existence domain of the classical solution.

The next lemma comes from Kong \cite{10}.
\begin{lemma}
If there exist two points $\alpha_{0}$ and $\beta_{0}$ satisfying (3.7), then the characteristic $x=x^{-}(t,\alpha_{0})$ must intersect $x=x^{+}(t,\beta_{0})$ in finite time, where we assume that the classical solution exists.
\end{lemma}

In what follows, under the assumptions (H1)-(H5), we consider the Cauchy problem given by (2.5), (2.7).

Let us fix $(\alpha_{0},\beta_{0})$ satisfying the assumptions (H1)-(H5), we introduce
$$
t_{0}=\int_{\alpha_{0}}^{\beta_{0}}\frac{1}{\Lambda_{+}(\zeta)-\Lambda_{-}(\zeta)}d\zeta
\eqno(3.12)
$$
and
$$
x_{0}=\frac{1}{2}\left\{\alpha_{0}+\beta_{0}+\int_{\alpha_{0}}^{\beta_{0}}\frac{\Lambda_{+}(\zeta)+\Lambda_{-}(\zeta)}{\Lambda_{+}(\zeta)-\Lambda_{-}(\zeta)}d\zeta\right\}.
\eqno(3.13)
$$
%
%
%

\begin{lemma}
There exists a positive constant $\epsilon$ such that $\alpha_{0}$ is the unique zero point of $f(\alpha)$, i.e., $$f(\alpha_{0})=0\:\: \rm{but}\:\:f(\alpha)\neq0,\:\:\quad\rm{for}\:\:\alpha\in(\alpha_{0}-\epsilon,\alpha_{0}+\epsilon).$$
\end{lemma}
\begin{proof}
The result comes from (3.8) and (3.9) directly.
\end{proof}
It is obvious that (2.10)-(2.11) define a mapping from the region $U\triangleq\{(\alpha,\beta)\mid\alpha\leq\beta\}$ to the domain $\{(t,x)\mid t\geq0,x\in \mathbb{R}\}$. Denote it by $\Pi:$
$$
\Pi(\alpha,\beta)=(t(\alpha,\beta),x(\alpha,\beta)).
\eqno(3.14)
$$
We introduce the Jacobian matrix of $\Pi$
$$
\mathbf{\bigtriangleup(\alpha,\beta)}=
\left(\begin{array} {cc}
t_{\alpha}&t_{\beta}\\
x_{\alpha}&x_{\beta} \end{array}
\right)
\eqno(3.15)
$$\\
and its Jacobian
$$
J(\alpha,\beta) = t_{\alpha}x_{\beta} - t_{\beta}x_{\alpha}.
\eqno(3.16)
$$
\begin{definition}
A point \textbf{p} in $U$ is called a regular point of the mapping $\Pi$
if the rank $\vartriangle$ is 2 at \textbf{p}. Otherwise, \textbf{p} is called a singular point of $\Pi$.
\end{definition}
It is easy to verify that $\textbf{p}$ is a singular point is equivalent to $\Lambda_{-}(\alpha)=\Lambda_{+}(\beta)$, which can form a smooth curve defined by an explicit function $\beta=\beta(\alpha)$, since $\Lambda_{+}^{'}(\beta)<0$.
\begin{definition}
Let \textbf{p} be a singular point of $\Pi$ and $\Upsilon(\alpha) = (\alpha,\beta(\alpha))$ be the parametric equation with $\Upsilon(\alpha_{0})=\textbf{p}$ for $J(\alpha,\beta)=0$. \textbf{p} is called a fold point of $\Pi$, if $\frac{d}{d\alpha} (\Pi\circ\Upsilon)(\alpha_{0}) \neq(0,0)$, and \textbf{p} is called a cusp point of $\Pi$, if $\frac{d}{d\alpha} (\Pi\circ\Upsilon)(\alpha_{0}) = (0,0)$ but $\frac{d^{2}}{d\alpha^{2}}(\Pi\circ\Upsilon)(\alpha_{0}) \neq (0,0).$
\end{definition}
 \begin{lemma}
 (A) The curve $\beta=\beta(\alpha)$ is strictly increasing as a function of $\alpha$;
 (B) the singular points $(\alpha,\beta)\neq(\alpha_{0},\beta_{0})$ are fold points, while $(\alpha_{0},\beta_{0})$ is a cusp point.
 \end{lemma}
\begin{proof} Differenting $\Lambda_{-}(\alpha)=\Lambda_{+}(\beta)$ with respect to $\alpha$ gives
 $$
 \Lambda^{'}_{-}(\alpha)=\Lambda^{'}_{+}(\beta)\beta_{\alpha},
 \eqno(3.18)
 $$
then
 $$
 \beta_{\alpha}=\frac{\Lambda^{'}_{-}(\alpha)}{\Lambda^{'}_{+}(\beta)}>0.
 \eqno(3.19)
 $$
Equation (3.19) implies that the curve $\beta=\beta(\alpha)$ is strictly increasing as a function of $\alpha$. This proves Part (A).

We next prove Part (B).

To do so, we notice that along $\beta=\beta(\alpha)$
 \begin{eqnarray}
 & &\frac{d}{d\alpha}(t(\alpha,\beta),x(\alpha,\beta))\nonumber\\
 &=&\frac{d}{d\alpha}\left(\int_{\alpha}^{\beta}\frac{1}{\Lambda_{+}(\zeta)-\Lambda_{-}(\zeta)}d\zeta , \frac{1}{2}\left\{\alpha+\beta+\int_{\alpha}^{\beta}\frac{\Lambda_{+}(\zeta)+\Lambda_{-}(\zeta)}{\Lambda_{+}(\zeta)-\Lambda_{-}(\zeta)}d\zeta\right\}\right)
 \nonumber\\
 &=&\left(\frac{\beta_{\alpha}}{\Lambda_{+}(\beta)-\Lambda_{-}(\beta)}-\frac{1}{\Lambda_{+}(\alpha)-\Lambda_{-}(\alpha)} , \frac{\Lambda_{+}(\beta)\beta_{\alpha}}
 {\Lambda_{+}(\beta)-\Lambda_{-}(\beta)}-\frac{\Lambda_{-}(\alpha)}{\Lambda_{+}(\alpha)-\Lambda_{-}(\alpha)}\right)\nonumber\\
 &=&\left(\frac{\Lambda^{'}_{-}(\alpha)}{(\Lambda_{+}(\beta)-\Lambda_{-}(\beta))\Lambda^{'}_{+}(\beta)}-\frac{1}{\Lambda_{+}(\alpha)-\Lambda_{-}(\alpha)}
  , \frac{\Lambda_{+}(\beta)\Lambda^{'}_{-}(\alpha)}{(\Lambda_{+}(\beta)-\Lambda_{-}(\beta))\Lambda^{'}_{+}(\beta)}-
 \frac{\Lambda_{+}(\beta)}{\Lambda_{+}(\alpha)-\Lambda_{-}(\alpha)}\right)\nonumber\\
  &=&\frac{1}{\Lambda^{'}_{+}(\beta)}\left(\frac{\Lambda^{'}_{-}(\alpha)}{\Lambda_{+}(\beta)-\Lambda_{-}(\beta)}-\frac{\Lambda^{'}_{+}(\beta)}{\Lambda_{+}(\alpha)-\Lambda_{-}(\alpha)},\Lambda_{+}(\beta)\left\{\frac{\Lambda^{'}_{-}(\alpha)}{\Lambda_{+}(\beta)-\Lambda_{-}(\beta)}-\frac{\Lambda^{'}_{+}(\beta)}{\Lambda_{+}(\alpha)-\Lambda_{-}(\alpha)}\right\}\right)\nonumber.
 \end{eqnarray}
The assumptions (H4)-(H5) yield
$$\frac{d}{d\alpha}(t(\alpha,\beta) , x(\alpha,\beta))\big|_{(\alpha,\beta)\neq(\alpha_{0},\beta_{0})} \neq 0.$$\\
On the other hand, along the curve $\beta=\beta(\alpha)$
 \begin{eqnarray}
 & &\frac{d^{2}}{d\alpha^{2}}(t(\alpha,\beta),x(\alpha,\beta))\nonumber\\
 &=&\frac{d}{d\alpha}\left(\frac{1}{\Lambda^{'}_{+}(\beta)}f(\alpha),\frac{\Lambda_{+}(\beta)}{\Lambda^{'}_{+}(\beta)}f(\alpha)\right)\nonumber\\
 &=&\left(-\frac{\Lambda^{''}_{+}(\beta)\beta_{\alpha}}{(\Lambda^{'}_{+}(\beta))^{2}}f(\alpha)+\frac{1}{\Lambda^{'}_{+}(\beta)}f^{'}(\alpha),
 \frac{[(\Lambda_{+}^{'}(\beta))^{2}-\Lambda_{+}(\beta)\Lambda^{''}_{+}(\beta)]\beta_{\alpha}}{(\Lambda^{'}_{+}(\beta))^{2}}f(\alpha)+\frac{\Lambda_{+}(\beta)}{\Lambda^{'}_{+}(\beta)}f^{'}(\alpha)\right).\nonumber
  \end{eqnarray}
The assumptions (H4)-(H5) again give
$$\frac{d}{d\alpha}(t(\alpha,\beta),x(\alpha,\beta))\big|_{(\alpha,\beta)=(\alpha_{0},\beta_{0})}=(0,0),\:\text{but}\:\frac{d^{2}}{d\alpha^{2}}(t(\alpha,\beta),x(\alpha,\beta))\big|_{(\alpha,\beta)=(\alpha_{0},\beta_{0})}\neq(0,0).$$
Thus, the singular points $(\alpha,\beta)\neq(\alpha_{0},\beta_{0})$ are fold points, while $(\alpha_{0},\beta_{0})$ is a cusp point.
\end{proof}
\begin{lemma}
 Under the assumptions (H1)-(H5), $t_{0}$ is the unique minimum point on the interval $(\alpha_{0}-\epsilon,\alpha_{0}+\epsilon)$, where $\epsilon$ is given in Lemma 3.4.
\end{lemma}
\begin{proof} From Lemma 3.4 and by a straightforward calculation we have
$$\frac{dt}{d\alpha}\bigg|_{\alpha=\alpha_{0}}=0,\quad \frac{d^{2}t}{d\alpha^{2}}\bigg|_{\alpha=\alpha_{0}}>0.
\eqno(3.20)$$
This proves the lemma.
\end{proof}
We next discuss the position and property of $\Pi(\alpha,\beta(\alpha))$ in the $(t,x)$-plane.

Introduce $\Upsilon_{l}$ as the graph of the curve $\beta=\beta(\alpha)$ with domain $(\alpha_{0}-\epsilon,\alpha_{0})$ and $\Upsilon_{r}$ as the graph of the curve $\beta=\beta(\alpha)$ with domain $(\alpha_{0},\alpha_{0}+\epsilon)$. Then we define $\Gamma_{l}=\Pi(\Upsilon_{l})\;\text{and}\;\Gamma_{r}=\Pi(\Upsilon_{r})$.
We have the following lemma.
\begin{lemma}Under the assumptions (H1)-(H5), $\Gamma_{l}$ and $\Gamma_{r}$ form a smooth curve in (t,x)-plane which can be defined by an explicit function $t=t(x)$, moreover, $\Gamma_{l}$ is increasing and concave with respect to $x$, $\Gamma_{r}$ is decreasing and concave with respect to $x$.
\end{lemma}
\begin{proof} By Lemma 3.5, we have
$$\frac{dt}{d\alpha}=\frac{f(\alpha)}{\Lambda^{'}_{+}(\beta)},\quad \frac{dx}{d\alpha}=\frac{\Lambda_{+}(\beta)}{\Lambda^{'}_{+}(\beta)}f(\alpha).
\eqno(3.21)
$$
According to (3.8) and (3.9),
$$
f(\alpha)>0,\quad\Lambda_{+}(\beta)>0,\quad\quad\forall\: \alpha\in(\alpha_{0}-\epsilon,\alpha_{0})
$$
and
$$
f(\alpha)<0,\quad\Lambda_{+}(\beta)<0,\quad\quad\forall\: \alpha\in(\alpha_{0},\alpha_{0}+\epsilon).
$$
So by (3.21) and (3.4),
$$
\frac{dt}{d\alpha}=\frac{f(\alpha)}{\Lambda^{'}_{+}(\beta)}<0,\quad \frac{dx}{d\alpha}=\frac{\Lambda_{+}(\beta)}{\Lambda^{'}_{+}(\beta)}f(\alpha)<0,\quad\forall\:
\alpha\in(\alpha_{0}-\epsilon,\alpha_{0}).
$$
By the implicit function theorem $\Gamma_{l}$ form a smooth curve $t=t(x)$. Moreover
$$\frac{dt}{dx}>0,
$$ so $\Gamma_{l}$ is increasing with respect to $x$.

On the other hand, for $\alpha\in(\alpha_{0},\alpha_{0}+\epsilon)$ it holds that
$$
\frac{dt}{d\alpha}=\frac{f(\alpha)}{\Lambda^{'}_{+}(\beta)}>0,\quad \frac{dx}{d\alpha}=\frac{\Lambda_{+}(\beta)}{\Lambda^{'}_{+}(\beta)}f(\alpha)<0.
$$
This gives
$$\frac{dt}{dx}<0.
$$ Thus, $\Gamma_{r}$ is decreasing with respect to $x$.

Moreover, since
$$
\frac{d^{2}t}{dx^{2}}=\frac{d}{dx}(\frac{dt}{dx})=\frac{d}{dx}(\frac{1}{\Lambda_{-}(\alpha)})=\frac{d}{d\alpha}(\frac{1}{\Lambda_{-}(\alpha)})\frac{d\alpha}{dx}
=-\frac{\Lambda_{-}^{'}(\alpha)\Lambda_{+}^{'}(\beta)}{\Lambda_{-}^{2}(\alpha)\Lambda_{+}(\beta)f(\alpha)}<0,
$$
we have
$$
\frac{d^{2}t}{dx^{2}}<0,\quad \forall\: \alpha\in(\alpha_{0}-\epsilon,\alpha_{0}).
$$
This implies that $\Gamma_{l}$ is concave with respect to $x$. Similarly, we have

$$
\frac{d^{2}t}{dx^{2}}<0,\quad \forall\:\alpha\in(\alpha_{0},\alpha_{0}+\epsilon),
$$
namely, $\Gamma_{r}$ is concave with respect to $x$.
\end{proof}
Based on the properties derived in Lemmas 3.5-3.7, we can sketch the map from $(\alpha,\beta)$ to $(t,x)$ (see Figure 4).
\begin{figure}[h]
\centering
\includegraphics[trim=0 0 0 0,scale=0.80]{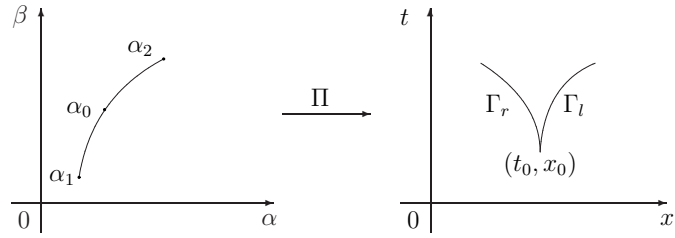}
\caption{ The mapping $\Pi$ under the assumptions (H1)-(H5). }
\end{figure}

\begin{remark}
From $(t_{0},x_{0})$, there exist only two characteristics which intersect the $x$-axis at $\alpha_{0}$ and $\beta_{0}$, respectively. Then, at $(t_{0},x_{0})$ it holds that $$\frac{dx}{dt}=\Lambda_{-}(\alpha_{0})=\Lambda_{+}(\beta_{0})=0,$$ namely, the two characteristics are tangent at $(t_{0},x_{0})$.
\end{remark}
Lemmas 2.1, 3.1, 3.5 and 3.6 lead to the following main result.
\begin{theorem}
Under the assumptions (H1)-(H5), the smooth solution of Cauchy problem (2.5) and (2.7) blows up at $(t_{0},x_{0})$, which is defined by (3.12)-(3.13), and $t_{0}$ is the blowup time. Furthermore, the blowup is geometric blowup.
\end{theorem}
\begin{remark}
The geometric blowup comes from Alinhac \cite{A}. Roughly speaking, the solution itself keeps bounded, however, the derivatives of first order go to infinity when $(t,x)$ tends to the blowup point.
\end{remark}
\begin{remark}
In the domain bounded by $\Gamma_{l}$ and $\Gamma_{r}$, characteristics of the same family must intersect, see Figure 5.
\end{remark}
\begin{figure}[h]
\centering
\includegraphics[trim=0 0 0 0,scale=0.80]{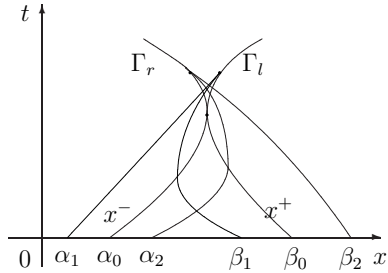}
 \caption{ The characteristics }
\end{figure}

\section{Estimates of singularities}
In this section we shall establish some estimates for the solution near the blowup point, these estimates describe the behavior of singularities near the blowup point. In what follows, we focus on the domain
$$
O_{\epsilon}=\{(t,x)\mid(t-t_{0})^{2}+(x-x_{0})^{2}<\epsilon^{2}\}.
$$
Let
$$\widetilde{t}= t-t_{0},\quad \widetilde{x}= x-x_{0},\quad \widetilde{\alpha}=\alpha-\alpha_{0},\quad \widetilde{\beta}=\beta-\beta_{0}.$$
\begin{remark}
Throughout the paper, without special notations, the above symbols are adopted to denote the differences of the vector components between regular points and the blowup point.
\end{remark}

We have the following theorem
\begin{theorem}
Under the assumptions (H1)-(H5), in the neighborhood of the blowup point, it holds that for any $ (t,x)\in O_{\epsilon}\setminus(t_{0},x_{0}) $

if $\widetilde{x}=o(|\widetilde{t}|^{\frac{3}{2}}),$
\begin{displaymath}
\left\{\begin{array}{ll}
|u(t,x)-u(t_{0},x_{0})|\leq F_{1}|\widetilde{t}|+F_{2}|\frac{\widetilde{x}}{\widetilde{t}}|, \vspace{2mm}\\
 |\rho|\leq F_{3}|\widetilde{t}|^{-1},\vspace{2mm}\\
|u_{x}|\leq  F_{4}|\widetilde{t}|^{-1},\vspace{2mm}\\
|u_{t}|\leq F_{5}+F_{6}\frac{|\widetilde{x}|}{|\widetilde{t}|^{2}},\vspace{2mm}\\
|\rho_{x}|\leq F_{7}|\widetilde{t}|^{-3},\vspace{2mm}\\
|\rho_{t}|\leq\frac{F_{8}}{|\widetilde{t}|^{2}}+F_{9}\frac{|\widetilde{x}|}{|\widetilde{t}|^{3}},\vspace{2mm}\\
\end{array}
\right.
\end{displaymath}

if $|\widetilde{t}|^{\frac{3}{2}}=o(\widetilde{x}),$
\begin{displaymath}
\left\{\begin{array}{ll}
|u(t,x)-u(t_{0},x_{0})|\leq F_{10}|\widetilde{x}|^{\frac{1}{3}},\vspace{2mm}\\
|\rho|\leq F_{11}|\widetilde{x}|^{-\frac{2}{3}},\vspace{2mm}\\
|u_{x}|\leq F_{12}|\widetilde{x}|^{-\frac{2}{3}},\vspace{2mm}\\
|u_{t}|\leq F_{13}|\widetilde{x}|^{-\frac{1}{3}},\vspace{2mm}\\
|\rho_{x}|\leq F_{14}|\widetilde{x}|^{-2},\vspace{2mm}\\
|\rho_{t}|\leq F_{15}|\widetilde{x}|^{-\frac{5}{3}},\vspace{2mm}\\
\end{array}
\right.
\end{displaymath}

if $\widetilde{x}=O(1)|\widetilde{t}|^{\frac{3}{2}},$
\begin{displaymath}
\left\{\begin{array}{ll}
|u(t,x)-u(t_{0},x_{0})|\leq F_{16}|\widetilde{x}|^{\frac{1}{3}},\vspace{2mm}\\
|\rho|\leq F_{17}|\widetilde{t}|^{-1},\vspace{2mm}\\
|u_{x}|\leq F_{18}|\widetilde{t}|^{-1},\vspace{2mm}\\
|u_{t}|\leq F_{19}|\widetilde{t}|^{-\frac{1}{2}},\vspace{2mm}\\
|\rho_{x}|\leq F_{20}|\widetilde{t}|^{-3},\vspace{2mm}\\
|\rho_{t}|\leq F_{21}|\widetilde{t}|^{-\frac{5}{2}},
\end{array}
\right.
\end{displaymath}
for sufficiently small $\widetilde{t}$ and $\widetilde{x}$, where $F_{i}$ $(i=1,2,\cdots,21)$ are positive constants depend only on the initial data at $(\alpha_{0},\beta_{0})$ and the symbol $O(1)$ denote a quantity whose absolute value is bounded depending on the relationship between $\widetilde{x}$ and $|\widetilde{t}|^{\frac{3}{2}}$ when $\widetilde{x}$ is sufficiently small. \end{theorem}

 It follows from (2.10), (2.11), (3.12) and (3.13) that
 \begin{align*}
\begin{split}
\widetilde{x}&=\frac{1}{2}\left\{\alpha-\alpha_{0}+\beta-\beta_{0}+\int_{\alpha}^{\beta}\frac{\Lambda_{+}(\zeta)+\Lambda_{-}(\zeta)}
{\Lambda_{+}(\zeta)-\Lambda_{-}(\zeta)}d\zeta-\int_{\alpha_{0}}^{\beta_{0}}\frac{\Lambda_{+}(\zeta)+\Lambda_{-}(\zeta)}
{\Lambda_{+}(\zeta)-\Lambda_{-}(\zeta)}d\zeta\right\}\\
&=\int_{\alpha_{0}}^{\alpha}\frac{-\Lambda_{-}(\zeta)}{\Lambda_{+}(\zeta)-\Lambda_{-}(\zeta)}d\zeta
+\int_{\beta_{0}}^{\beta}\frac{\Lambda_{+}(\zeta)}{\Lambda_{+}(\zeta)-\Lambda_{-}(\zeta)}d\zeta\\
&=\frac{\Lambda_{+}(\beta_{0})}{\Lambda_{+}(\beta_{0})-\Lambda_{-}(\beta_{0})}\widetilde{\beta}-
\frac{\Lambda_{-}(\alpha_{0})\widetilde{\alpha}}{\Lambda_{+}(\alpha_{0})-\Lambda_{-}(\alpha_{0})}\\
&\quad+\frac{1}{2}\left(\frac{\Lambda^{'}_{+}(\beta_{0})}{\Lambda_{+}(\beta_{0})-\Lambda_{-}(\beta_{0})}-
\frac{\Lambda_{+}(\beta_{0})(\Lambda^{'}_{+}(\beta_{0})-\Lambda^{'}_{-}(\beta_{0}))}
{(\Lambda_{+}(\beta_{0})-\Lambda_{-}(\beta_{0}))^{2}}\right)\widetilde{\beta}^{2}\\
&\quad-\frac{1}{2}
\left(\frac{\Lambda^{'}_{-}(\alpha_{0})}{\Lambda_{+}(\alpha_{0})-\Lambda_{-}(\alpha_{0})}-\frac{(\Lambda^{'}_{+}(\alpha_{0})-\Lambda^{'}_{-}(\alpha_{0}))\Lambda_{-}(\alpha_{0})}
{(\Lambda_{+}(\alpha_{0})-\Lambda_{-}(\alpha_{0}))^{2}}\right)\widetilde{\alpha}^{2}\\
&\quad+\frac{1}{6}\left[\frac{\Lambda^{''}_{+}(\beta_{0})}{\Lambda_{+}(\beta_{0})-\Lambda_{-}(\beta_{0})}
-\frac{\Lambda^{'}_{+}(\beta_{0})\left(\Lambda^{'}_{+}(\beta_{0})-\Lambda^{'}_{-}(\beta_{0})\right)}{(\Lambda_{+}(\beta_{0})-\Lambda_{-}(\beta_{0}))^{2}}\right.\\
&\quad-\frac{\Lambda^{'}_{+}(\beta_{0})\left(\Lambda^{'}_{+}(\beta_{0})-\Lambda^{'}_{-}(\beta_{0})\right)
+\Lambda_{+}(\beta_{0})\left(\Lambda^{''}_{+}(\beta_{0})-\Lambda^{''}_{-}(\beta_{0})\right)}{(\Lambda_{+}(\beta_{0})-\Lambda_{-}(\beta_{0}))^{2}}
\\
&\quad\left.+\frac{2\Lambda_{+}(\beta_{0})\left(\Lambda^{'}_{+}(\beta_{0})-\Lambda^{'}_{-}(\beta_{0})\right)^{2}}{(\Lambda_{+}(\beta_{0})-\Lambda_{-}(\beta_{0}))^{3}}\right]\widetilde{\beta}^{3}
\\
&\quad-\frac{1}{6}\left[\frac{\Lambda^{''}_{-}(\alpha_{0})}{\Lambda_{+}(\alpha_{0})-\Lambda_{-}(\alpha_{0})}-\frac{2\Lambda^{'}_{-}(\alpha_{0})\left(\Lambda^{'}_{+}(\alpha_{0})-\Lambda^{'}_{-}(\alpha_{0})\right)}
{(\Lambda_{+}(\alpha_{0})-\Lambda_{-}(\alpha_{0}))^{2}}\right.\\
&\left.\quad-\frac{\Lambda_{+}(\alpha_{0})\left(\Lambda^{''}_{+}(\alpha_{0})-\Lambda^{''}_{-}(\alpha_{0})\right)}{(\Lambda_{+}(\alpha_{0})-\Lambda_{-}(\alpha_{0}))^{2}}
+\frac{2\Lambda_{-}(\alpha_{0})\left(\Lambda^{'}_{+}(\alpha_{0})-\Lambda'_{-}(\alpha_{0})\right)^{2}}{(\Lambda_{+}(\alpha_{0})-\Lambda_{-}(\alpha_{0}))^{3}}\right]\widetilde{\alpha}^{3}
\end{split} \tag{4.1}
\end{align*}
and
\begin{align*}
\begin{split}
\widetilde{t}&=\int_{\alpha}^{\beta}\frac{1}{\Lambda_{+}(\zeta)-\Lambda_{-}(\zeta)}d\zeta-\int_{\alpha_{0}}^{\beta_{0}}\frac{1}{\Lambda_{+}(\zeta)-\Lambda_{-}(\zeta)}d\zeta\\
&=\int_{\alpha}^{\alpha_{0}}\frac{1}{\Lambda_{+}(\zeta)-\Lambda_{-}(\zeta)}d\zeta+\int_{\beta_{0}}^{\beta}\frac{1}{\Lambda_{+}(\zeta)-\Lambda_{-}(\zeta)}d\zeta\\
&=\frac{\widetilde{\beta}}{\Lambda_{+}(\beta_{0})-\Lambda_{-}(\beta_{0})}-\frac{1}{2}\frac{\Lambda^{'}_{+}(\beta_{0})-\Lambda^{'}_{-}(\beta_{0})}
{(\Lambda_{+}(\beta_{0})-\Lambda_{-}(\beta_{0}))^{2}}\widetilde{\beta}^{2}\\
&\quad+ \frac{1}{6}\left(\frac{\Lambda^{''}_{-}(\beta_{0})
-\Lambda^{''}_{+}(\beta_{0})}{(\Lambda_{+}(\beta_{0})-\Lambda_{-}(\beta_{0}))^{2}}
+\frac{2(\Lambda^{'}_{+}(\beta_{0})-\Lambda^{'}_{-}(\beta_{0}))^{2}}
{(\Lambda_{+}(\beta_{0})-\Lambda_{-}(\beta_{0}))^{3}}\right)\widetilde{\beta}^{3}\\
&\quad- \left(\frac{\widetilde{\alpha}}{\Lambda_{+}(\alpha_{0})-\Lambda_{-}(\alpha_{0})}
-\frac{1}{2}\frac{\Lambda^{'}_{+}(\alpha_{0})-\Lambda^{'}_{-}(\alpha_{0})}
{(\Lambda_{+}(\alpha_{0})-\Lambda_{-}(\alpha_{0}))^{2}}\widetilde{\alpha}^{2}\right)\\
&\quad - \frac{1}{6}\left(\frac{\Lambda^{''}_{-}(\alpha_{0})-\Lambda^{''}_{+}(\alpha_{0})}{(\Lambda_{+}(\alpha_{0})-\Lambda_{-}(\alpha_{0}))^{2}}+
\frac{2(\Lambda^{'}_{+}(\alpha_{0})-\Lambda^{'}_{-}(\alpha_{0}))^{2}}{(\Lambda_{+}(\alpha_{0})-\Lambda_{-}(\alpha_{0}))^{3}}\right)\widetilde{\alpha}^{3}.
\end{split}  \tag{4.2}
\end{align*}

To prove Theorem 4.1, we need the following lemmas.
\begin{lemma}
Under the assumptions (H1)-(H5), it holds that
\begin{align*}
\begin{split}
\widetilde{\alpha}&=\left(-\frac{1}{2}(C_{2}\widetilde{t}^{2}+C_{3}\widetilde{x})+\left(\frac{1}{4}(C_{2}\widetilde{t}^{2}+C_{3}\widetilde{x})^{2}
+\frac{C_{1}^{3}\widetilde{t}^{3}}{27}\right)^{\frac{1}{2}}\right)^{\frac{1}{3}}\\
& \quad+\left(-\frac{1}{2}(C_{2}\widetilde{t}^{2}+C_{3}\widetilde{x})-\left(\frac{1}{4}(C_{2}\widetilde{t}^{2}+C_{3}\widetilde{x})^{2}
+\frac{C_{1}^{3}\widetilde{t}^{3}}{27}\right)^{\frac{1}{2}}\right)^{\frac{1}{3}},
\end{split}\tag{4.3}
\end{align*}
where $C_{i}$ ($i=1,2,3$) depend only on the values of initial data at $(\alpha_{0},\beta_{0})$.
\end{lemma}
\begin{proof} Under the assumptions (H1)-(H5), it holds that
$$
\Lambda_{-}(\alpha_{0})=\Lambda_{+}(\beta_{0})=0.
\eqno(4.4)
$$
By (3.11)
\begin{align*}
\begin{split}
f^{'}(\alpha_{0})&=\frac{\Lambda^{''}_{-}(\alpha_{0})\Lambda^{2}_{+}(\alpha_{0})-\Lambda^{2}_{-}(\beta_{0})\Lambda^{''}_{+}(\beta_{0})}
{-\Lambda_{-}(\beta_{0})\Lambda^{2}_{+}(\alpha_{0})}\\
&\quad+\frac{
\Lambda^{'}_{+}(\beta_{0})\Lambda_{-}(\beta_{0})\left[\Lambda^{'}_{+}(\beta_{0})-\Lambda^{'}_{-}(\beta_{0})-\left(\Lambda^{'}_{+}(\alpha_{0})-\Lambda^{'}_{-}(\alpha_{0})\right)\right]}
{-\Lambda_{-}(\beta_{0})\Lambda^{2}_{+}(\alpha_{0})}\\
&<0.
\end{split}\tag{4.5}
\end{align*}
So
$$\begin{array}{l}
\Lambda^{''}_{-}(\alpha_{0})\Lambda^{2}_{+}(\alpha_{0})-\Lambda^{2}_{-}(\beta_{0})\Lambda^{''}_{+}(\beta_{0})\vspace{2mm}\\
\quad+
\Lambda^{'}_{+}(\beta_{0})\Lambda_{-}(\beta_{0})\left[\Lambda^{'}_{+}(\beta_{0})-\Lambda^{'}_{-}(\beta_{0})
-\left(\Lambda^{'}_{+}(\alpha_{0})-\Lambda^{'}_{-}(\alpha_{0})\right)\right]<0.
\end{array}
\eqno(4.6)
$$
By (4.6), it suffices to expand $\tilde{t}$ and $\tilde{x}$ up to third order of $\tilde{\alpha}$ and $\tilde{\beta}$ to get an optimal estimate in (4.1) and (4.2). Noting the assumption (3.10) and using (4.1) and (4.2) lead to
\begin{align*}
\begin{split}
\widetilde{t}&=-\frac{\widetilde{\beta}}{\Lambda_{-}(\beta_{0})}-\frac{1}{2}\frac{\Lambda^{'}_{+}(\beta_{0})-\Lambda^{'}_{-}(\beta_{0})}
{\Lambda^{2}_{-}(\beta_{0})}\widetilde{\beta}^{2}\nonumber\\
&\quad+\frac{1}{6}\left(\frac{\Lambda^{''}_{-}(\beta_{0})-\Lambda^{''}_{+}(\beta_{0})}{\Lambda^{2}_{-}(\beta_{0})}
-\frac{2\left(\Lambda^{'}_{+}(\beta_{0})-\Lambda^{'}_{-}(\beta_{0})\right)^{2}}{\Lambda^{3}_{-}(\beta_{0})}\right)\widetilde{\beta}^{3}\nonumber\\
&\quad-
\frac{\widetilde{\alpha}}{\Lambda_{+}(\alpha_{0})}
+\frac{1}{2}\frac{\Lambda^{'}_{+}(\alpha_{0})-\Lambda^{'}_{-}(\alpha_{0})}
{\Lambda^{2}_{+}(\alpha_{0})}\widetilde{\alpha}^{2}\nonumber\\
& \quad-\frac{1}{6}\left(\frac{\Lambda^{''}_{-}(\alpha_{0})-\Lambda^{''}_{+}(\alpha_{0})}{\Lambda^{2}_{+}(\alpha_{0})}
+\frac{2(\Lambda^{'}_{+}(\alpha_{0})-\Lambda^{'}_{-}(\alpha_{0}))^{2}}{\Lambda^{3}_{+}(\alpha_{0})}\right)\widetilde{\alpha}^{3}\nonumber
\end{split}\tag{4.7}
\end{align*}
and
\begin{align*}
\begin{split}
\widetilde{x}&=-\frac{\Lambda^{'}_{+}(\beta_{0})}{2\Lambda_{-}(\beta_{0})}\widetilde{\beta}^{2}
+\frac{1}{6}\left(\frac{\Lambda^{''}_{+}(\beta_{0})}{-\Lambda_{-}(\beta_{0})}-\frac{2\Lambda^{'}_{+}(\beta_{0})
\left(\Lambda^{'}_{+}(\beta_{0})-\Lambda^{'}_{-}(\beta_{0})\right)}{\Lambda^{2}_{-}(\beta_{0})}\right)\widetilde{\beta}^{3}\nonumber\\
&\quad-\frac{\Lambda^{'}_{-}(\alpha_{0})}{2\Lambda_{+}(\alpha_{0})}\widetilde{\alpha}^{2}
-\frac{1}{6}
\left(\frac{\Lambda^{''}_{-}(\alpha_{0})}{\Lambda_{+}(\alpha_{0})}-\frac{2\Lambda^{'}_{-}(\alpha_{0})
\left(\Lambda^{'}_{+}(\alpha_{0})-\Lambda^{'}_{-}(\alpha_{0})\right)}{\Lambda^{2}_{+}(\alpha_{0})}\right)\widetilde{\alpha}^{3}.\nonumber
\end{split}\tag{4.8}
\end{align*}
Noting (4.7), (4.8) and using the iterative method, we can obtain
$$
\widetilde{x}=B_{3}\widetilde{\alpha}^{3}+B_{2}\widetilde{t}\widetilde{\alpha}+B_{1}\widetilde{t}^{2},
\eqno(4.9)
$$
where
$$
B_{1}=-\frac{\Lambda^{'}_{+}(\beta_{0})\Lambda_{-}(\beta_{0})}{2}<0,\quad B_{2}=\Lambda^{'}_{-}(\alpha_{0})<0
\eqno(4.10)
$$
and
\begin{align*}
\begin{split}
B_{3}&=-\frac{\Lambda^{''}_{-}(\alpha_{0})\Lambda^{2}_{+}(\alpha_{0})-\Lambda^{2}_{-}(\beta_{0})\Lambda^{''}_{+}(\beta_{0})
}{6\Lambda^{3}_{+}(\alpha_{0})}\\
&\quad+\frac{\Lambda^{'}_{+}(\beta_{0})\Lambda_{-}(\beta_{0})[\Lambda^{'}_{+}(\beta_{0})-\Lambda^{'}_{-}(\beta_{0})
-(\Lambda^{'}_{+}(\alpha_{0})-\Lambda^{'}_{-}(\alpha_{0}))]}{6\Lambda^{3}_{+}(\alpha_{0})}\\
&>0.
\end{split} \tag{4.11}
\end{align*}
Solving equation (4.9) gives
\begin{align*}
\begin{split}
\widetilde{\alpha}&=\left(-\frac{1}{2}(C_{2}\widetilde{t}^{2}+C_{3}\widetilde{x})+\left(\frac{1}{4}(C_{2}\widetilde{t}^{2}+C_{3}\widetilde{x})^{2}
+\frac{C_{1}^{3}\widetilde{t}^{3}}{27}\right)^{\frac{1}{2}}\right)^{\frac{1}{3}}\nonumber\\
& \quad+\left(-\frac{1}{2}(C_{2}\widetilde{t}^{2}+C_{3}\widetilde{x})-\left(\frac{1}{4}(C_{2}\widetilde{t}^{2}+C_{3}\widetilde{x})^{2}
+\frac{C_{1}^{3}\widetilde{t}^{3}}{27}\right)^{\frac{1}{2}}\right)^{\frac{1}{3}},\nonumber
\end{split}
\end{align*}
where $$C_{1}=\frac{B_{2}}{B_{3}}<0,\quad C_{2}=\frac{B_{1}}{B_{3}}<0,\quad C_{3}=-\frac{1}{B_{3}}<0.$$
By (4.10) and (4.11), we observe that $C_{i}$ ($i=1,2,3$) depend only on the value of initial data at $(\alpha_{0},\beta_{0})$.
Thus, the lemma is proved.
\end{proof}
\begin{lemma}
Under the assumptions (H1)-(H5), it holds that
\begin{displaymath}
\widetilde{\alpha} = \left\{\begin{array}{ll}
-\frac{C_{2}}{C_{1}}\widetilde{t}-\frac{C_{3}\widetilde{x}}{C_{1}\widetilde{t}},&\textrm{$\widetilde{x}=o(|\widetilde{t}|^{\frac{3}{2}})$},\vspace{2mm}\\
(-C_{3}\widetilde{x})^{\frac{1}{3}},&\textrm{$|\widetilde{t}|^{\frac{3}{2}}=o(\widetilde{x})$},\vspace{2mm}\\
C(-\frac{1}{2}(C_{3}\widetilde{x}))^{\frac{1}{3}},&\textrm{$\widetilde{x}=O(1)|\widetilde{t}|^{\frac{3}{2}}$}
\end{array}\right.
\eqno(4.12)
\end{displaymath}
for $\widetilde{t}$ and $\widetilde{x}$ sufficiently small, $C$ stands for a constant and $C_{i}\; (i=1,2,3)$ are determined by Lemma 4.1.
\end{lemma}
\begin{proof} By Lemma 4.1, for simplicity, we may rewrite $(4.3)$ as
$$
\widetilde{\alpha}=\left(A+(A^{2}+B)^{\frac{1}{2}}\right)^{\frac{1}{3}}+\left(A-(A^{2}+B)^{\frac{1}{2}}\right)^{\frac{1}{3}},
\eqno(4.13)
$$
where $$A=-\frac{1}{2}(C_{2}\widetilde{t}^{2}+C_{3}\widetilde{x}) \quad\text{and} \quad B=\frac{C_{1}^{3}\widetilde{t}^{3}}{27}.$$

We next prove Lemma 4.2 by distinguishing three cases:

\textbf{Case I}:\: $A^{2}=o(B)$, i.e.,\: $\widetilde{x}=o(|\widetilde{t}|^{\frac{3}{2}})$.

By Taylor expansion, we have
\begin{align*}
\begin{split}
\widetilde{\alpha}&=\left[A+B^{\frac{1}{2}}\left(1+\frac{A^{2}}{B}\right)^{\frac{1}{2}}\right]^{\frac{1}{3}}+\left[A-B^{\frac{1}{2}}\left(1+\frac{A^{2}}{B}\right)^{\frac{1}{2}}\right]^{\frac{1}{3}}\nonumber\\
&=\left[A+B^{\frac{1}{2}}\left(1+\frac{A^{2}}{2B}+o(\frac{A^{2}}{B})\right)\right]^{\frac{1}{3}}+\left[A-B^{\frac{1}{2}}\left(1+\frac{A^{2}}{2B}+o(\frac{A^{2}}{B})\right)\right]^{\frac{1}{3}}\nonumber\\
&=B^{\frac{1}{6}}\left\{\left[\frac{A}{B^{\frac{1}{2}}}+(1+\frac{A^{2}}{2B}+o(\frac{A^{2}}{B}))\right]^{\frac{1}{3}}-\left[-\frac{A}{B^{\frac{1}{2}}}+(1+\frac{A^{2}}{2B}+o(\frac{A^{2}}{B}))\right]^{\frac{1}{3}}\right\}\nonumber\\
&=B^{\frac{1}{6}}\left[1+\frac{A^{2}}{6B}+\frac{A}{3B^{\frac{1}{2}}}-(1+\frac{A^{2}}{6B}-\frac{A}{3B^{\frac{1}{2}}})+o(\frac{A^{2}}{B})\right]\nonumber\\
&=B^{\frac{1}{6}}(\frac{2A}{3B^{\frac{1}{2}}}+o(\frac{A^{2}}{B}))=\frac{2AB^{-\frac{1}{3}}}{3}+o(A^{2}B^{-\frac{5}{6}})\nonumber\\
&=-\frac{C_{2}}{C_{1}}\widetilde{t}-\frac{C_{3}\widetilde{x}}{C_{1}\widetilde{t}}+o(\tilde{t}^{\frac{3}{2}}+\tilde{x}^{2}\tilde{t}^{-\frac{5}{2}}+\tilde{x}\tilde{t}^{-\frac{1}{2}}).
\end{split}
\end{align*}
The special case $-\frac{C_{2}}{C_{1}}\widetilde{t}-\frac{C_{3}\widetilde{x}}{C_{1}\widetilde{t}}=0$ implies that $\widetilde{\alpha}=o(\tilde{t})$, which does not affect the main results of the paper, so we do not distinguish this special case anymore.

\textbf{Case II}:\: $B=o(A^{2})$, i.e.,\: $|\widetilde{t}|^{\frac{3}{2}}=o(\widetilde{x})$.

By Taylor expansion, we have
\begin{align*}
\begin{split}
\widetilde{\alpha}&=\left[A+|A|\left(1+\frac{B}{A^{2}}\right)^{\frac{1}{2}}\right]^{\frac{1}{3}}+\left[A-|A|\left(1+\frac{B}{A^{2}}\right)^{\frac{1}{2}}\right]^{\frac{1}{3}}\nonumber\\
&=\left[A+|A|\left(1+\frac{B}{2A^{2}}+o(\frac{B}{A^{2}})\right)\right]^{\frac{1}{3}}+\left[A-|A|\left(1+\frac{B}{2A^{2}}+o(\frac{B}{A^{2}})\right)\right]^{\frac{1}{3}}\nonumber\\
&=|A|^{\frac{1}{3}}\left\{\left[sign(A)+\left(1+\frac{B}{2A^{2}}+o(\frac{B}{A^{2}})\right)\right]^{\frac{1}{3}}+\left[sign(A)-\left(1+\frac{B}{2A^{2}}+o(\frac{B}{A^{2}})\right)\right]^{\frac{1}{3}}\right\}\nonumber\\
&=\left((2+o(1))A\right)^{\frac{1}{3}}=(-C_{2}\widetilde{t}^{2}-C_{3}\widetilde{x})^{\frac{1}{3}}
=(-C_{3}\widetilde{x})^{\frac{1}{3}}(1+\frac{C_{2}\tilde{t}^{2}}{C_{3}\tilde{x}})^{\frac{1}{3}}\nonumber\\
&=(-C_{3}\widetilde{x})^{\frac{1}{3}}
(1+\frac{C_{2}\tilde{t}^{2}}{3C_{3}\tilde{x}}+o(\frac{\tilde{t}^{2}}{\tilde{x}}))
=(-C_{3}\widetilde{x})^{\frac{1}{3}}+o(\tilde{x}^{\frac{1}{3}}).\nonumber
\end{split}
\end{align*}

\textbf{Case III}: \:$B=\overline{O(1)}(A^{2})$, i.e., \:$\widetilde{x}=O(1)|\widetilde{t}|^{\frac{3}{2}}$.

By Taylor expansion, we have
\begin{align*}
\begin{split}
\widetilde{\alpha}&=\left[A+\left(A^{2}+\overline{O(1)}A^{2}\right)^{\frac{1}{2}}\right]^{\frac{1}{3}}+\left[A-\left(A^{2}+\overline{O(1)}A^{2}\right)^{\frac{1}{2}}\right]^{\frac{1}{3}}\nonumber\\
&=\left[A+|A|\left(1+\overline{O(1)}\right)^{\frac{1}{2}}\right]^{\frac{1}{3}}+\left[A-|A|\left(1+\overline{O(1)}\right)^{\frac{1}{2}}\right]^{\frac{1}{3}}\nonumber\\
&=\left\{\left[1+\left(1+\overline{O(1)}\right)^{\frac{1}{2}}\right]^{\frac{1}{3}}+\left[1-\left(1+\overline{O(1)}\right)^{\frac{1}{2}}\right]^{\frac{1}{3}}\right\}A^{\frac{1}{3}}\nonumber\\
&=\left\{\left[1+\left(1+\overline{O(1)}\right)^{\frac{1}{2}}\right]^{\frac{1}{3}}+\left[1-\left(1+\overline{O(1)}\right)^{\frac{1}{2}}\right]^{\frac{1}{3}}\right\}\left[-\frac{1}{2}(C_{2}\widetilde{t}^{2}+C_{3}\widetilde{x})\right]^{\frac{1}{3}}\nonumber\\
&\triangleq C\left(-\frac{1}{2}(C_{2}\widetilde{t}^{2}+C_{3}\widetilde{x})\right)^{\frac{1}{3}}=C\left(-\frac{1}{2}C_{3}\widetilde{x}\right)^{\frac{1}{3}}+o(\tilde{x}^{\frac{1}{3}}),\nonumber
\end{split}
\end{align*}
where $$C=\left(1+\left(1+\overline{O(1)}\right)^{\frac{1}{2}}\right)^{\frac{1}{3}}+\left(1-\left(1+\overline{O(1)}\right)^{\frac{1}{2}}\right)^{\frac{1}{3}},\quad \overline{O(1)}=\frac{4C_{1}^{3}sign(\widetilde{t})}{27C_{3}^{2}(O(1))^{2}}.\eqno(4.14)$$
Since the constants derived in the proof are not equal to zero, we discard the higher order terms.
\end{proof}
\begin{lemma}
Under the assumptions (H1)-(H5) and (4.7), it holds that
\begin{displaymath}
\Lambda_{+}(\beta)-\Lambda_{-}(\alpha) = \left\{\begin{array}{ll}
D_{1}\widetilde{t},&\textrm{$\widetilde{x}=o(|\widetilde{t}|^{\frac{3}{2}})$},\vspace{2mm}\\
D_{2}\widetilde{x}^{\frac{2}{3}},&\textrm{$|\widetilde{t}|^{\frac{3}{2}}=o(\widetilde{x})$},\vspace{2mm}\\
D_{3}\widetilde{t},&\textrm{$\widetilde{x}=O(1)|\widetilde{t}|^{\frac{3}{2}}$},
\end{array}\right.
\eqno(4.15)
\end{displaymath}
where $D_{i}\,(i=1,2,3)$ depend only on the initial data at $(\alpha_{0},\beta_{0})$.
\end{lemma}
\begin{proof} Iterating (4.7) two times and retain $\tilde{\alpha}$ to second order term, we obtain
$$
\tilde{\beta}=-\Lambda_{-}(\beta_{0})\left[\tilde{t}+\frac{\tilde{\alpha}}{\Lambda_{+}(\alpha_{0})}+
\frac{\Lambda_{+}^{'}(\beta_{0})-\Lambda_{-}^{'}(\beta_{0})-(\Lambda_{+}^{'}(\alpha_{0})-\Lambda_{-}^{'}(\alpha_{0}))}
{2\Lambda_{+}^{2}(\alpha_{0})}\tilde{\alpha}^{2}\right].
$$
Then by assumptions (H1)-(H5) and above discussions, we have
\begin{align*}
\begin{split}
&\Lambda_{+}(\beta)-\Lambda_{-}(\alpha)\nonumber\\
&=\Lambda_{+}(\beta)-\Lambda_{+}(\beta_{0})+\Lambda_{-}(\alpha_{0})-\Lambda_{-}(\alpha)\nonumber\\
&=\Lambda^{'}_{+}(\beta_{0})\widetilde{\beta}-\Lambda^{'}_{-}(\alpha_{0})\widetilde{\alpha}+\frac{1}{2}\Lambda^{''}_{+}(\beta_{0})
\widetilde{\beta}^{2}-\frac{1}{2}\Lambda^{''}_{-}(\alpha_{0})\widetilde{\alpha}^{2}\nonumber\\
&=-\Lambda^{'}_{+}(\beta_{0})\Lambda_{-}(\beta_{0})\widetilde{t}+M\widetilde{\alpha}^{2},\nonumber
\end{split}\tag{4.16}
\end{align*}
where
$$
M=\frac{\Lambda_{+}^{''}(\beta_{0})\Lambda_{-}^{2}(\beta_{0})-\Lambda_{-}^{''}(\alpha_{0})\Lambda_{+}^{2}(\alpha_{0})
-\Lambda_{+}^{'}(\beta_{0})\Lambda_{-}(\beta_{0})\left[
\Lambda_{+}^{'}(\beta_{0})-\Lambda_{-}^{'}(\beta_{0})-(\Lambda_{+}^{'}(\alpha_{0})-\Lambda_{-}^{'}(\alpha_{0}))\right]}{2\Lambda_{+}^{2}(\alpha_{0})}.
$$
By (3.11), it holds that $M\neq0$, thus, by Lemma 4.2, we have

\textbf{Case I}: $\widetilde{x}=o(|\widetilde{t}|^{\frac{3}{2}})$.
\begin{align*}
\begin{split}
\Lambda_{+}(\beta)-\Lambda_{-}(\alpha)&=-\Lambda^{'}_{+}(\beta_{0})\Lambda_{-}(\beta_{0})\widetilde{t}+M\left(-\frac{C_{2}}{C_{1}}\widetilde{t}-\frac{C_{3}\widetilde{x}}{C_{1}\widetilde{t}}\right)^{2}\nonumber\\
&=-\Lambda^{'}_{+}(\beta_{0})\Lambda_{-}(\beta_{0})\widetilde{t}+M\frac{C^{2}_{3}\widetilde{x}^{2}}{C^{2}_{1}\widetilde{t}^{2}}\nonumber\\
&=-\Lambda^{'}_{+}(\beta_{0})\Lambda_{-}(\beta_{0})\widetilde{t}\triangleq D_{1}\widetilde{t},\nonumber
\end{split}\tag{4.17}
\end{align*}
where $$D_{1}=-\Lambda^{'}_{+}(\beta_{0})\Lambda_{-}(\beta_{0}).$$

\textbf{Case II}: $|\widetilde{t}|^{\frac{3}{2}}=o(\widetilde{x})$.
\begin{align*}
\begin{split}
\Lambda_{+}(\beta)-\Lambda_{-}(\alpha)&=-\Lambda^{'}_{+}(\beta_{0})\Lambda_{-}(\beta_{0})\widetilde{t}+M(-C_{3}\widetilde{x})^{\frac{2}{3}}\nonumber\\
&=M(-C_{3}\widetilde{x})^{\frac{2}{3}}\triangleq D_{2}\widetilde{x}^{\frac{2}{3}},\nonumber
\end{split}\tag{4.18}\end{align*}
where $$D_{2}=MC_{3}^{\frac{2}{3}}.$$

\textbf{Case III}: $\widetilde{x}=O(1)|\widetilde{t}|^{\frac{3}{2}}$.
\begin{align*}
\begin{split}
\Lambda_{+}(\beta)-\Lambda_{-}(\alpha)&=-\Lambda^{'}_{+}(\beta_{0})\Lambda_{-}(\beta_{0})\widetilde{t}+MC^{2}\left(-\frac{1}{2}(C_{3}\widetilde{x})\right)^{\frac{2}{3}}\nonumber\\
&\triangleq D_{3}\widetilde{t},\nonumber
\end{split}\tag{4.19}\end{align*}
where $$D_{3}=-\Lambda^{'}_{+}(\beta_{0})\Lambda_{-}(\beta_{0})+MC^{2}\left(\frac{C_{3}O(1)}{2}\right)^{\frac{2}{3}}sign(\widetilde{t}).$$
Thus, the lemma is proved.
\end{proof}

\textbf{Proof of Theorem 4.1}.\\
 By (2.4), we have
$$
u(t,x)=\frac{\lambda_{+}(t,x)+\lambda_{-}(t,x)}{2}=\frac{\Lambda_{+}(\beta)+\Lambda_{-}(\alpha)}{2}
\eqno(4.20)
$$
and
$$
\rho(t,x)=\frac{2\mu}{\lambda_{+}(t,x)-\lambda_{-}(t,x)}=\frac{2\mu}{\Lambda_{+}(\beta)-\Lambda_{-}(\alpha)}.
\eqno(4.21)
$$
Here and hereafter, we use $(\alpha,\beta)$ (resp. $(\alpha_{0},\beta_{0})$) to denote the characteristic parameters defined by (2.10) and (2.11) corresponding to $(t,x)$ (resp. $(t_{0},x_{0})$).

In order to estimate $u(t,x)$ at the blowup point $(t_{0},x_{0})$, we firstly estimate
$$\Lambda_{+}(\beta)-\Lambda_{+}(\beta_{0})+\Lambda_{-}(\alpha)-\Lambda_{-}(\alpha_{0}).
\eqno(4.22)
$$
By Taylor expansion and (4.7),
\begin{align*}
\begin{split}
&\quad\Lambda_{+}(\beta)-\Lambda_{+}(\beta_{0})+\Lambda_{-}(\alpha)-\Lambda_{-}(\alpha_{0})\nonumber\\
&=\Lambda^{'}_{+}(\beta_{0})\widetilde{\beta}+\Lambda^{'}_{-}(\alpha_{0})\widetilde{\alpha}\nonumber\\
&=\Lambda^{'}_{+}(\beta_{0})\left(-\Lambda_{-}(\beta_{0})\widetilde{t}-\frac{\Lambda_{-}(\beta_{0})}{\Lambda_{+}(\alpha_{0})}\widetilde{\alpha}\right)+\Lambda^{'}_{-}(\alpha_{0})\widetilde{\alpha}\nonumber\\
&=-\Lambda^{'}_{+}(\beta_{0})\Lambda_{-}(\beta_{0})\widetilde{t}+\frac{-\Lambda^{'}_{+}(\beta_{0})\Lambda_{-}(\beta_{0})+\Lambda^{'}_{-}(\alpha_{0})\Lambda_{+}(\alpha_{0})}{\Lambda_{+}(\alpha_{0})}\widetilde{\alpha}\nonumber\\
&=-\Lambda^{'}_{+}(\beta_{0})\Lambda_{-}(\beta_{0})\widetilde{t}-\frac{2\Lambda^{'}_{+}(\beta_{0})\Lambda_{-}(\beta_{0})}{\Lambda_{+}(\alpha_{0})}\widetilde{\alpha}.\nonumber
\end{split}\tag{4.23}\end{align*}
By Lemma 4.2, we have

\textbf{Case I}: $\widetilde{x}=o(|\widetilde{t}|^{\frac{3}{2}}).$

We have
\begin{align*}
\begin{split}
&\quad\Lambda_{+}(\beta)+\Lambda_{-}(\alpha)\\
&=-\Lambda^{'}_{+}(\beta_{0})\Lambda_{-}(\beta_{0})\widetilde{t}-\frac{2\Lambda^{'}_{+}(\beta_{0})\Lambda_{-}(\beta_{0})}{\Lambda_{+}(\alpha_{0})}\left(-\frac{C_{2}}{C_{1}}\widetilde{t}-\frac{C_{3}\widetilde{x}}{C_{1}\widetilde{t}}\right)\nonumber\\
&=\left(-\Lambda^{'}_{+}(\beta_{0})\Lambda_{-}(\beta_{0})+\frac{2\Lambda^{'}_{+}(\beta_{0})\Lambda_{-}(\beta_{0})C_{2}}{\Lambda_{+}(\alpha_{0})C_{1}}\right)\widetilde{t}
+\frac{2\Lambda^{'}_{+}(\beta_{0})\Lambda_{-}(\beta_{0})C_{3}}{\Lambda_{+}(\alpha_{0})C_{1}}\frac{\widetilde{x}}{\widetilde{t}}\nonumber\\
&\triangleq C_{5}\widetilde{t}+C_{6}\frac{\widetilde{x}}{\widetilde{t}}=C_{6}\frac{\widetilde{x}}{\widetilde{t}}+o(\tilde{t}),\nonumber
\end{split}
\tag{4.24}
\end{align*}
where $$C_{5}=-\Lambda^{'}_{+}(\beta_{0})\Lambda_{-}(\beta_{0})+\frac{2\Lambda^{'}_{+}(\beta_{0})\Lambda_{-}(\beta_{0})C_{2}}{\Lambda_{+}(\alpha_{0})C_{1}}=0,\quad C_{6}=\frac{2\Lambda^{'}_{+}(\beta_{0})\Lambda_{-}(\beta_{0})C_{3}}{\Lambda_{+}(\alpha_{0})C_{1}}.$$

\textbf{Case II}: $|\widetilde{t}|^{\frac{3}{2}}=o(\widetilde{x}).$

 In this case, we get
\begin{align*}
\begin{split}
\Lambda_{+}(\beta)+\Lambda_{-}(\alpha)&=-\Lambda^{'}_{+}(\beta_{0})\Lambda_{-}(\beta_{0})\widetilde{t}-\frac{2\Lambda^{'}_{+}(\beta_{0})\Lambda_{-}(\beta_{0})}{\Lambda_{+}(\alpha_{0})}(-C_{3}\widetilde{x})^{\frac{1}{3}}\nonumber\\
&=-\frac{2\Lambda^{'}_{+}(\beta_{0})\Lambda_{-}(\beta_{0})}{\Lambda_{+}(\alpha_{0})}(-C_{3}\widetilde{x})^{\frac{1}{3}}\triangleq C_{7}\widetilde{x}^{\frac{1}{3}},\nonumber
\end{split}\tag{4.25}\end{align*}
where $$C_{7}=\frac{2\Lambda^{'}_{+}(\beta_{0})\Lambda_{-}(\beta_{0})}{\Lambda_{+}(\alpha_{0})}C_{3}^{\frac{1}{3}}.$$

\textbf{Case III}: $\widetilde{x}=O(1)|\widetilde{t}|^{\frac{3}{2}}.$

We obtain
\begin{align*}
\begin{split}
\Lambda_{+}(\beta)+\Lambda_{-}(\alpha)&=-\Lambda^{'}_{+}(\beta_{0})\Lambda_{-}(\beta_{0})\widetilde{t}-\frac{2\Lambda^{'}_{+}(\beta_{0})\Lambda_{-}(\beta_{0})}{\Lambda_{+}(\alpha_{0})}C\left(-\frac{1}{2}C_{3}\widetilde{x}\right)^{\frac{1}{3}}\nonumber\\
&=-\frac{2\Lambda^{'}_{+}(\beta_{0})\Lambda_{-}(\beta_{0})}{\Lambda_{+}(\alpha_{0})}C\left(-\frac{1}{2}C_{3}\widetilde{x}\right)^{\frac{1}{3}}\triangleq C_{8}\widetilde{x}^{\frac{1}{3}},\nonumber
\end{split}\tag{4.26}\end{align*}
where $$C_{8}=\frac{2\Lambda^{'}_{+}(\beta_{0})\Lambda_{-}(\beta_{0})}{\Lambda_{+}(\alpha_{0})}C\left(\frac{1}{2}C_{3}\right)^{\frac{1}{3}}.$$

In order to estimate $u_{x}$, $\rho_{x}$, $u_{t}$ and $\rho_{t}$,
  we have to estimate $(\lambda_{+})_{x}$, $(\lambda_{-})_{x}$, $(\lambda_{+})_{t}$, $(\lambda_{-})_{t}$.

By (3.2), we can obtain the estimates on $(\lambda_{+})_{x}$ and $(\lambda_{-})_{x}$.

It follows from the system (2.5) that
\begin{align*}
\begin{split}
(\lambda_{+})_{t}&=-\lambda_{-}(\lambda_{+})_{x}=-\Lambda_{-}(\alpha)\Lambda^{'}_{+}(\beta)\frac{\Lambda_{+}(\beta)-\Lambda_{-}(\beta)}{\Lambda_{+}(\beta)-\Lambda_{-}(\alpha)}\nonumber\\
&=-\Lambda^{'}_{-}(\alpha_{0})\Lambda^{'}_{+}(\beta)\frac{\Lambda_{+}(\beta)-\Lambda_{-}(\beta)}{\Lambda_{+}(\beta)-\Lambda_{-}(\alpha)}\widetilde{\alpha}\nonumber
\end{split}\tag{4.29}\end{align*}
and
\begin{align*}
\begin{split}
(\lambda_{-})_{t}&=-\lambda_{+}(\lambda_{-})_{x}=-\Lambda_{+}(\beta)\Lambda^{'}_{-}(\alpha)\frac{\Lambda_{+}(\alpha)-\Lambda_{-}(\alpha)}{\Lambda_{+}(\beta)-\Lambda_{-}(\alpha)}\nonumber\\
&=-\Lambda^{'}_{+}(\beta_{0})\Lambda^{'}_{-}(\alpha)\frac{\Lambda_{+}(\alpha)-\Lambda_{-}(\alpha)}{\Lambda_{+}(\beta)-\Lambda_{-}(\alpha)}\widetilde{\beta}\nonumber\\
&=\Lambda_{-}(\beta_{0})\Lambda^{'}_{+}(\beta_{0})\Lambda^{'}_{-}(\alpha)\frac{\Lambda_{+}(\alpha)-\Lambda_{-}(\alpha)}{\Lambda_{+}(\beta)-\Lambda_{-}(\alpha)}\left(\widetilde{t}+\frac{\widetilde{\alpha}}{\Lambda_{+}(\alpha_{0})}\right).\nonumber
\end{split}\tag{4.30}\end{align*}

We now estimate $u_{x}$, $u_{t}$, $\rho_{x}$ and $\rho_{t}$

For $u_{x}$, noting (4.20)-(4.21) and (3.2), by direct calculations, we have
$$
u_{x}=\frac{(\lambda_{+})_{x}+(\lambda_{-})_{x}}{2}\triangleq \frac{M_{5}}{\Lambda_{+}(\beta)-\Lambda_{-}(\alpha)},
\eqno(4.31)
$$
where $$M_{5}=\frac{\Lambda^{'}_{+}(\beta)(\Lambda_{+}(\beta)-\Lambda_{-}(\beta))}{2}+\frac{\Lambda^{'}_{-}(\alpha)(\Lambda_{+}(\alpha)-\Lambda_{-}(\alpha))}{2}.$$
Similarly, for $u_{t}$, by (4.29) and (4.30), we have
$$
u_{t}=\frac{(\lambda_{+})_{t}+(\lambda_{-})_{t}}{2}\triangleq\frac{M_{6}\widetilde{t}+M_{7}\widetilde{\alpha}}{\Lambda_{+}(\beta)-\Lambda_{-}(\alpha)},
\eqno(4.32)
$$
where $$M_{6}=\Lambda^{'}_{+}(\beta_{0})\Lambda^{'}_{-}(\alpha_{0})\Lambda_{-}(\beta_{0})(\Lambda_{+}(\alpha)-\Lambda_{-}(\alpha))$$
and
$$M_{7}=\frac{\Lambda^{'}_{+}(\beta_{0})\Lambda^{'}_{-}(\alpha_{0})\Lambda_{-}(\beta_{0})\left(\Lambda_{+}(\alpha)-\Lambda_{-}(\alpha)\right)}{\Lambda_{+}(\alpha_{0})}
-\Lambda^{'}_{-}(\alpha_{0})\Lambda^{'}_{+}(\beta)\left(\Lambda_{+}(\beta)-\Lambda_{-}(\beta)\right).$$
For $\rho_{x}$, we have
$$
\rho_{x}=-\frac{2\mu\left((\lambda_{+})_{x}-(\lambda_{-})_{x}\right)}{(\lambda_{+}-\lambda_{-})^{2}}\triangleq \frac{M_{8}}{\left(\Lambda_{+}(\beta)-\Lambda_{-}(\alpha)\right)^{3}},
\eqno(4.33)
$$
where $$M_{8}=2\mu\left(\Lambda^{'}_{+}(\beta)\left(\Lambda_{+}(\beta)-\Lambda_{-}(\beta)\right)-\Lambda^{'}_{-}(\alpha)\left(\Lambda_{+}(\alpha)-\Lambda_{-}(\alpha)\right)\right).$$
For $\rho_{t}$, we have
$$
\rho_{t}=-\frac{2\mu\left((\lambda_{+})_{t}-(\lambda_{-})_{t}\right)}{(\lambda_{+}-\lambda_{-})^{2}}\triangleq
\frac{M_{9}\widetilde{t}+M_{10}\widetilde{\alpha}}{\left(\Lambda_{+}(\beta)-\Lambda_{-}(\alpha)\right)^{3}}
\eqno(4.34)
$$
where $$M_{9}=2\mu\Lambda^{'}_{+}(\beta_{0})\Lambda^{'}_{-}(\alpha_{0})\Lambda_{-}(\beta_{0})\left(\Lambda_{+}(\alpha)-\Lambda_{-}(\alpha)\right)$$
and
$$M_{10}=2\mu\left[\frac{\Lambda^{'}_{+}(\beta_{0})\Lambda^{'}_{-}(\alpha)\Lambda_{-}(\beta_{0})\left(\Lambda_{+}(\alpha)-\Lambda_{-}(\alpha)\right)}{\Lambda_{+}(\alpha_{0})}
+\Lambda^{'}_{-}(\alpha_{0})\Lambda^{'}_{+}(\beta)\left(\Lambda_{+}(\beta)-\Lambda_{-}(\beta)\right)\right].
$$
Then, by (4.20)-(4.21) and (4.31)-(4.34), we have

\textbf{Case I}: $\widetilde{x}=o(|\widetilde{t}|^{\frac{3}{2}}).$

It holds that
$$
|u(t,x)-u(t_{0},x_{0})|\leq |C_{5}||\widetilde{t}|+|C_{6}|\left|\frac{\widetilde{x}}{\widetilde{t}}\right|\triangleq F_{1}|\widetilde{t}|+F_{2}\left|\frac{\widetilde{x}}{\widetilde{t}}\right|,
\eqno(4.35)
$$
$$
|u_{x}|\leq 2| \frac{M_{5}}{\Lambda_{+}(\beta)-\Lambda_{-}(\alpha)}|\leq4|M_{5}|\frac{1}{|-\Lambda^{'}_{+}(\beta_{0})\Lambda_{-}(\beta_{0})\widetilde{t}|}
=\frac{4|M_{5}|}{\Lambda^{'}_{+}(\beta_{0})\Lambda_{-}(\beta_{0})}|\widetilde{t}|^{-1}\triangleq F_{4}|\widetilde{t}|^{-1},
\eqno(4.36)
$$
$$
|u_{t}|\leq2\frac{|M_{6}||\widetilde{t}|+|M_{7}||\widetilde{\alpha}|}{|D_{1}\widetilde{t}|}\triangleq F_{5}+F_{6}\frac{|\widetilde{x}|}{|\widetilde{t}|^{2}},
\eqno(4.37)
$$
$$
|\rho|\leq\frac{4\mu}{\Lambda^{'}_{+}(\beta_{0})\Lambda_{-}(\beta_{0})}|\widetilde{t}|^{-1}\triangleq F_{3}|\widetilde{t}|^{-1},
\eqno(4.38)
$$
$$
|\rho_{x}|\leq\frac{4|M_{8}|}{\left(\Lambda^{'}_{+}(\beta_{0})\Lambda_{-}(\beta_{0})\right)^{3}}|\widetilde{t}|^{-3}\triangleq F_{7}|\widetilde{t}|^{-3},
\eqno(4.39)
$$
and
$$
|\rho_{t}|\leq2\frac{|M_{9}||\widetilde{t}|+|M_{10}||\widetilde{\alpha}|}{|D_{1}\widetilde{t}|^{3}}\triangleq
\frac{F_{8}}{|\widetilde{t}|^{2}}+F_{9}\frac{|\widetilde{x}|}{|\widetilde{t}|^{4}}.
\eqno(4.40)
$$

\textbf{Case II}: $|\widetilde{t}|^{\frac{3}{2}}=o(\widetilde{x}).$

We have
$$
|u(t,x)-u(t_{0},x_{0})|\leq 2|C_{7}||\widetilde{x}^{\frac{1}{3}}|\triangleq F_{10}|\widetilde{x}|^{\frac{1}{3}},
\eqno(4.41)
$$
$$
|u_{x}|\leq\frac{2|M_{5}|}{|D_{2}\tilde{x}^{\frac{2}{3}}|}
\triangleq F_{12}|\widetilde{x}|^{-\frac{2}{3}},
\eqno(4.42)
$$
$$
|u_{t}|\leq2\frac{|M_{6}\widetilde{t}|+|M_{7}\widetilde{\alpha}|}{(D_{2}\widetilde{x})^{\frac{2}{3}}}\triangleq F_{13}|\widetilde{x}|^{-\frac{1}{3}},
\eqno(4.43)
$$
$$
|\rho|\leq\frac{4\mu}{D_{2}}(\widetilde{x})^{-\frac{2}{3}}\triangleq F_{11}|\widetilde{x}|^{-\frac{2}{3}},
\eqno(4.44)
$$
$$
|\rho_{x}|\leq\frac{2|M_{8}|}{D_{2}^{3}}|\widetilde{x}|^{-2}\triangleq F_{14}|\widetilde{x}|^{-2},
\eqno(4.45)
$$
and
$$
|\rho_{t}|\leq2\frac{|M_{9}\widetilde{t}|+|M_{10}\widetilde{\alpha}|}{(D_{2}\widetilde{x})^{2}}\triangleq F_{15}|\widetilde{x}|^{-\frac{5}{3}}.
\eqno(4.46)
$$

\textbf{Case III}: $\widetilde{x}=O(1)|\widetilde{t}|^{\frac{3}{2}}.$

 We obtain
$$
|u(t,x)-u(t_{0},x_{0})|\leq |C_{8}||\widetilde{x}|^{\frac{1}{3}}\triangleq F_{16}|\widetilde{x}|^{\frac{1}{3}},
\eqno(4.47)
$$
$$
|u_{x}|\leq F_{18}|\widetilde{t}|^{-1},
\eqno(4.48)
$$
$$
|u_{t}|\leq2\frac{|M_{6}\widetilde{t}|+|M_{7}C(\frac{1}{2}C_{3}O(1))||\widetilde{t}|^{\frac{1}{2}}}{|D_{3}\widetilde{t}|}\triangleq
N_{25}+N_{26}|\widetilde{t}|^{-\frac{1}{2}}=F_{19}|\widetilde{t}|^{-\frac{1}{2}},
\eqno(4.49)
$$
$$
|\rho|\leq F_{17}|\widetilde{t}|^{-1},
\eqno(4.50)
$$
$$
|\rho_{x}|\leq F_{20}|\widetilde{t}|^{-3},
\eqno(4.51)
$$
and
$$
|\rho_{t}|\leq2\frac{|M_{6}\widetilde{t}|+|M_{7}C(\frac{1}{2}C_{3}O(1))||\widetilde{t}|^{\frac{1}{2}}}{|D_{3}\widetilde{t}|^{3}}\triangleq
N_{31}|\widetilde{t}|^{-2}+N_{32}|\widetilde{t}|^{-\frac{5}{2}}=F_{21}|\widetilde{t}|^{-\frac{5}{2}}.
\eqno(4.52)
$$
Thus, the proof of Theorem 4.1 is completed.
\begin{remark}
From the above discussions, it is easy to say that the constants derived in above estimates are not equal to zero for  initial data satisfying assumptions (H1)-(H5).
\end{remark}
\section{analysis of singularity}
In this section we shall construct physical solutions with new kind of singularity for the system $(2.1)$. To do so, we firstly recall the traditional definition of weak solution. For simplicity, consider the conservation law
$$
u_{t}+f(u)_{x}=0,\quad t>0,\quad x\in\mathbb R,
\eqno(5.1)
$$
with initial data
$$
u(0,x)=u_{0}(x),\quad x\in\mathbb R,
$$
where $u(t,x)=(u_{1},\cdots,u_{n})(t,x)\in\mathbb R^{n},\,n\geq1$ and $f(u)$ is the flux vector-valued function in some open set $\Omega\subset\mathbb R^{n}$.
\begin{definition}
A bounded measurable function $u(t,x)$ is called a weak solution of the Cauchy problem (5.1) with bounded and measurable initial data $u_{0}$, provided that
\newcommand\diff{\,{\mathrm d}}
\[
\iint\limits_{t\geq0} (u\phi_{t}+f(u)\phi_{x}) \diff x \diff t+\int_{t=0}u_{0}\phi dx=0
\eqno(5.2)
\]
holds for all $\phi\in C_{0}^{1}(\mathbb{R^{+}}\times\mathbb{R})$, where $C_{0}^{1}$ denotes the class of $C^{1}$ functions $\phi$, which vanish outside a compact subset in $t\geq0$.
\end{definition}
In this paper, we generalize the above definition as follows
\begin{definition}
A measurable function $u(t,x)$ is called a weak solution of the Cauchy problem (5.1) with bounded and measurable initial data $u_{0}$, provided that (5.2) holds for all $\phi\in C_{0}^{1}(\mathbb{R^{+}}\times\mathbb{R}).$
\end{definition}
\begin{cor}
If $u$ is a weak solution, then it holds that:
\newcommand\diff{\,{\mathrm d}}
\[
\lim\limits_{\epsilon\rightarrow0}\iint\limits_{(t,x)\in D_{\epsilon}} (u\phi_{t}+f(u)\phi_{x}) \diff x \diff t=0,
\eqno(5.3)
\]
where $$D_{\epsilon}=\{(t,x)\mid |t-t_{0}|\leq\epsilon , |x-x_{0}|\leq\epsilon\}$$and $(t_{0},x_{0})$ is a blowup point.
\end{cor}
\begin{proof} If (5.2) holds, then taking $\phi\in C_{0}^{1}(D_{\epsilon})$ gives
 \[
\iint\limits_{(t,x)\in D_{\epsilon}} (u\phi_{t}+f(u)\phi_{x}) dxdt=0.
\] This is nothing but (5.3)
\end{proof}
Let
$$
u^{\pm}=u(t,x(t)\pm0),
$$
where $x(t)$ is a smooth curve across which $u$ has a jump discontinuity.
As in the traditional sense, we can also get the Rankine-Hugoniot condition (see \cite{14})
$$
s[u]=[f(u)],
\eqno(5.4)
$$
where $s=\frac{dx(t)}{dt}$ is the speed of discontinuity, $[u]=u^{+}-u^{-}$, the jump across $x(t)$ and similarly, $[f]=f(u^{+})-f(u^{-})$, at which we do not require that $u$ has well-defined limits on both sides of $x=x(t)$, i.e., $u$ may be infinity on either side of the discontinuity $x=x(t)$.
\begin{remark}
In our definition, we do not require that $u(t,x)$ is bounded everywhere, while, we need the singular integral in the left hand side of $(5.3)$ is convergent.\end{remark}
\begin{definition}
$u=u(t,x)$ of system (5.1) is said to be the ``Delta-like solution'', if it satisfies Definition 5.2 and $u(t,x)$ is smooth except on some points or curves or other domains, on which $u=\infty$.
\end{definition}
\begin{remark}
By Definition 5.3, the ``Delta-like'' solution is different from the ``shock-wave'' solution. Here, the density of system (2.1) is unbounded.
\end{remark}
\begin{lemma}
For system (2.1), if $x=x(t)$ is a curve of discontinuity and $\rho$ has a jump across $x(t)$, define $\rho$ on both sides of $x=x(t)$ as $\rho^{\pm}=\rho(t,x(t)\pm0)$. Then
$$
\frac{dx(t)}{dt}=u^{+}=u^{-}.
\eqno(5.5)
$$
where $u^{\pm}=u(t,x(t)\pm0)$ are the right and left limits, respectively.
\end{lemma}
\begin{proof}
If $x=x(t)$ is a curve of discontinuity of system (2.1), then, by (5.4), we have
$$
s(\rho^{+}-\rho^{-})=\rho^{+}u^{+}-\rho^{-}u^{-}
\eqno(5.6)
$$
and
$$
s(\rho^{+}u^{+}-\rho^{-}u^{-})=\rho^{+}(u^{+})^{2}-\rho^{-}(u^{-})^{2}.
\eqno(5.7)
$$
Assume $\rho^{+}\neq\rho^{-}$. Then by (5.6) and (5.7), we have
$$(\rho^{+}-\rho^{-})(\rho^{+}(u^{+})^{2}-\rho^{-}(u^{-})^{2})=(\rho^{+}u^{+}-\rho^{-}u^{-})^{2}.$$
By a simple calculation, we get
$$
u^{+}=u^{-},
$$
thus, by (5.6), we obtain $$
\frac{dx(t)}{dt}=u^{+}=u^{-}.
$$
\end{proof}
\begin{remark}
Since on both sides of the discontinuity $x=x(t)$, the pressure $p=0$, (5.7) holds accordingly.
\end{remark}
\begin{theorem}
Under the assumptions (H1)-(H5) in Section 3, the solution of the Cauchy problem (2.1), (2.6) constructed by the method of characteristics satisfy (5.3) in the strip $\{(t,x)\mid t\in[0,t_{0}],x\in\mathbb{R}\}$ and $t_{0}$ is defined by (3.10).
\end{theorem}
\begin{proof} It suffices to check that
\newcommand\diff{\,{\mathrm d}}
\[
\lim\limits_{\epsilon\rightarrow0}\iint\limits_{(t,x)\in D^{-}_{\epsilon}} (\rho\phi_{t}+(\rho u)\phi_{x}) \diff x \diff t =0
\eqno(5.8)
\]
and
\[
\lim\limits_{\epsilon\rightarrow0}\iint\limits_{(t,x)\in D^{-}_{\epsilon}} ((\rho u)\phi_{t}+(\rho u^{2})\phi_{x}) \diff x \diff t=0,
\eqno(5.9)
\]
where
$$
D^{-}_{\epsilon}=\{(t,x)\mid|x-x_{0}|\leq\epsilon,\quad t_{0}-\epsilon\leq t\leq t_{0}\}
$$

By Lemma 4.1, we have $$C_{1}<0.$$ We prove Theorem 5.1 by distinguishing the following three possible cases: $$B=O(1)A^{2},\quad B=o(A^{2}) \quad\text{and}\quad A^{2}=o(B),$$
where $A$ and $B$ are defined in the proof of Lemma 4.2.
From $B=A^{2}$, we have
$$
\frac{C_{1}^{3}\widetilde{t}^{3}}{27}=\frac{1}{4}(C_{2}\widetilde{t}^{2}+C_{3}\widetilde{x})^{2}.
\eqno(5.10)
$$ This implies that $$\widetilde{t}\leq0.$$ Equation (5.10) defines two curves passing through $(t_{0},x_{0})$ read
$$
\widetilde{x}=\frac{2C_{1}(-C_{1})^{\frac{1}{2}}}{(27)^{\frac{1}{2}}C_{3}}(-\widetilde{t})^{\frac{3}{2}}\triangleq G_{1}(-\widetilde{t})^{\frac{3}{2}}
\eqno(5.11)
$$
and
$$
\widetilde{x}=-\frac{2C_{1}(-C_{1})^{\frac{1}{2}}}{(27)^{\frac{1}{2}}C_{3}}(-\widetilde{t})^{\frac{3}{2}}\triangleq -G_{1}(-\widetilde{t})^{\frac{3}{2}},
\eqno(5.12)
$$
respectively. So we can break $D^{-}_{\epsilon}$ into $T_{1}$ and $T_{2}$ defined by
$$T_{1}=\{(t,x)\in D^{-}_{\epsilon}\mid B\geq A^{2}\}$$
and
$$T_{2}=\{(t,x)\in D^{-}_{\epsilon}\mid B<A^{2}\}.$$
Thus, it suffices to prove that (5.8) and (5.9) hold in $T_{1}\bigcup T_{2}$.

Case $A^{2}=o(B)$, namely,  $\widetilde{x}=o(|\widetilde{t}|^{\frac{3}{2}})$

By Theorem 4.1, we have the following asymptotic solutions:
$$\rho \approx k_{1}\widetilde{t}^{-1},
\eqno(5.13)
$$
$$u \approx k_{2}\widetilde{t}+k_{3}\frac{\widetilde{x}}{\widetilde{t}},
\eqno(5.14)
$$
so
 $$
 \rho u \approx m_{1}+m_{2}\widetilde{x}\tilde{t}^{-2},
 \eqno(5.15)
 $$
 $$
 \rho u^{2} \approx m_{3}\widetilde{t}+m_{4}\widetilde{x}^{2}\widetilde{t}^{-2}.
 \eqno(5.16)
 $$

Case $B=O(1)A^{2}$, namely, $\widetilde{x}=O(1)|\widetilde{t}|^{\frac{3}{2}}.$

We obtain
$$
\rho \approx k_{4}\widetilde{t}^{-1}=k_{5}\widetilde{x}^{-\frac{2}{3}},
\eqno(5.17)
$$
$$u \approx k_{6}\widetilde{x}^{\frac{1}{3}},
\eqno(5.18)
$$
so
$$
\rho u \approx m_{5}\widetilde{x}^{-\frac{1}{3}},
\eqno(5.19)
$$
$$
\rho u^{2}\approx m_{6}.
\eqno(5.20)
$$

Case $B=o(A^{2})$, namely, $|\widetilde{t}|^{\frac{3}{2}}=o(\widetilde{x})$.

By Theorem 4.1, we have the following asymptotic solutions:
$$
\rho \approx k_{7}\widetilde{x}^{-\frac{2}{3}},
\eqno(5.21)
$$
$$u \approx k_{8}\widetilde{x}^{\frac{1}{3}},
\eqno(5.22)
$$
so
$$
\rho u \approx m_{7}\widetilde{x}^{-\frac{1}{3}},
\eqno(5.23)
$$
$$
\rho u^{2} \approx m_{8},
\eqno(5.24)
$$
where $k_{i}$ $(i=1,\cdots,8)$ and $m_{i}$ $(i=1,\cdots,8)$ are constants depending only on the initial data at $(\alpha_{0},\beta_{0})$.
By (5.8) and (5.9), it follows that
\begin{align*}
\begin{split}
P&=\lim\limits_{\epsilon\rightarrow0}\iint\limits_{(t,x)\in D^{-}_{\epsilon}} (\rho\phi_{t}+(\rho u)\phi_{x}) \diff x \diff t\nonumber\\
 &=\lim\limits_{\epsilon\rightarrow0}\iint\limits_{(t,x)\in  T_{1}} (\rho\phi_{t}+(\rho u)\phi_{x}) \diff x \diff t
 +\lim\limits_{\epsilon\rightarrow0}\iint\limits_{(t,x)\in  T_{2}} (\rho\phi_{t}+(\rho u)\phi_{x}) \diff x \diff t\nonumber\\
&\triangleq P_{1}+P_{2}
\end{split}\tag{5.25}
\end{align*}
and
\begin{align*}
\begin{split}
Q&=\lim\limits_{\epsilon\rightarrow0}\iint\limits_{(t,x)\in D_{\epsilon}} ((\rho u)\phi_{t}+(\rho u^{2})\phi_{x}) \diff x \diff t\nonumber\\
&=\lim\limits_{\epsilon\rightarrow0}\iint\limits_{(t,x)\in  T_{1}} ((\rho u)\phi_{t}+(\rho u^{2}+p)\phi_{x}) \diff x \diff t+\lim\limits_{\epsilon\rightarrow0}\iint\limits_{(t,x)\in  T_{2}} ((\rho u)\phi_{t}+(\rho u^{2}+p)\phi_{x}) \diff x \diff t\nonumber\\
&\triangleq Q_{1}+Q_{2}.
\end{split}\tag{5.26}
\end{align*}
We next prove  $$|P_{i}|=0\: (i=1,2) \quad \text{and} \quad|Q_{i}|=0\:(i=1,2).$$ Define
$$K_{1}=\{(t,x)\mid\widetilde{x}=o(|\widetilde{t}|^{\frac{3}{2}})\},$$
$$K_{2}= \{(t,x)\mid\widetilde{x}=O(1)|\widetilde{t}|^{\frac{3}{2}}\}$$
and
$$K_{3}= \{(t,x)\mid|\widetilde{t}|^{\frac{3}{2}}=o(\widetilde{x})\}.$$
Then by (5.13)-(5.24), we obtain
\begin{eqnarray}
|P_{1}|&\leq&\lim\limits_{\epsilon\rightarrow0}\iint\limits_{(t,x)\in  T_{1}} (|\rho||\phi_{t}|+|(\rho u)||\phi_{x}|) \diff x \diff t\nonumber\\
&\leq&\max|\phi_{t}|\lim\limits_{\epsilon\rightarrow0}\iint\limits_{(t,x)\in  T_{1}} |\rho| \diff x
\diff t+\max|\phi_{x}|\lim\limits_{\epsilon\rightarrow0}\iint\limits_{(t,x)\in  T_{1}} |\rho u| \diff x\diff t\nonumber\\
&\leq&\max|\phi_{t}|\lim\limits_{\epsilon\rightarrow0}\iint\limits_{(t,x)\in T_{1}\bigcap K_{1}} |\rho| \diff x
\diff t+\max|\phi_{x}|\lim\limits_{\epsilon\rightarrow0}\iint\limits_{(t,x)\in  T_{1}\bigcap K_{1}} |\rho u| \diff x\diff t\nonumber\\
& &+\max|\phi_{t}|\lim\limits_{\epsilon\rightarrow0}\iint\limits_{(t,x)\in T_{1}\bigcap K_{2}} |\rho| \diff x
\diff t+\max|\phi_{x}|\lim\limits_{\epsilon\rightarrow0}\iint\limits_{(t,x)\in T_{1}\bigcap K_{2}} |\rho u| \diff x\diff t\nonumber\\
&\leq&\max|\phi_{t}|\lim\limits_{\epsilon\rightarrow0}\iint\limits_{(t,x)\in T_{1}\bigcap K_{1}} |k_{1}\widetilde{t}^{-1}| \diff x
\diff t+\max|\phi_{x}|\lim\limits_{\epsilon\rightarrow0}\iint\limits_{(t,x)\in  T_{1}\bigcap K_{1}} |m_{1}+m_{2}\widetilde{x}\tilde{t}^{-2}| \diff x\diff t\nonumber\\
& &+\max|\phi_{t}|\lim\limits_{\epsilon\rightarrow0}\iint\limits_{(t,x)\in  T_{1}\bigcap K_{2}} |k_{5}\widetilde{x}^{-\frac{2}3{}}| \diff x
\diff t+\max|\phi_{x}|\lim\limits_{\epsilon\rightarrow0}\iint\limits_{(t,x)\in  T_{1}\bigcap K_{2}} |m_{5}\widetilde{x}^{-\frac{1}{3}}| \diff x\diff t\nonumber\\
&\leq&\max|\phi_{t}|\lim\limits_{\epsilon\rightarrow0}\int_{-G_{1}(-\widetilde{\tau})^{\frac{3}{2}}}^{G_{1}(-\widetilde{\tau})^{\frac{3}{2}}}\int_{-\epsilon}^{0}(|k_{1}\widetilde{x}^{-\frac{2}{3}}|+
|k_{5}\widetilde{x}^{-\frac{2}{3}}|)d\widetilde{t}d\widetilde{x}\nonumber\\
& &+\max|\phi_{x}|\lim\limits_{\epsilon\rightarrow0}\int_{-G_{1}(-\widetilde{\tau})^{\frac{3}{2}}}^{G_{1}(-\widetilde{\tau})^{\frac{3}{2}}}\int_{-\epsilon}^{0}(|m_{1}+m_{2}\widetilde{x}\tilde{t}^{-2}|+
|m_{5}\widetilde{x}^{-\frac{1}{3}}| )d\widetilde{t}d\widetilde{x}\nonumber\\
&=&0.\nonumber
\end{eqnarray}
\begin{eqnarray}
|P_{2}|&\leq&\lim\limits_{\epsilon\rightarrow0}\iint\limits_{(t,x)\in  T_{2}} (|\rho||\phi_{t}|+|(\rho u)||\phi_{x}|) \diff x \diff t\nonumber\\
&\leq&\max|\phi_{t}|\lim\limits_{\epsilon\rightarrow0}\iint\limits_{(t,x)\in T_{2}} |\rho| \diff x
\diff t+\max|\phi_{x}|\lim\limits_{\epsilon\rightarrow0}\iint\limits_{(t,x)\in T_{2}} |\rho u| \diff x\diff t\nonumber\\
&\leq&\max|\phi_{t}|\lim\limits_{\epsilon\rightarrow0}\iint\limits_{(t,x)\in T_{2}\bigcap K_{3}} |\rho| \diff x
\diff t+\max|\phi_{x}|\lim\limits_{\epsilon\rightarrow0}\iint\limits_{(t,x)\in  T_{2}\bigcap K_{3}} |\rho u| \diff x\diff t\nonumber\\
& &+\max|\phi_{t}|\lim\limits_{\epsilon\rightarrow0}\iint\limits_{(t,x)\in  T_{2}\bigcap K_{2}} |\rho| \diff x
\diff t+\max|\phi_{x}|\lim\limits_{\epsilon\rightarrow0}\iint\limits_{(t,x)\in T_{2}\bigcap K_{2}} |\rho u| \diff x\diff t\nonumber\\
&\leq&\max|\phi_{t}|\lim\limits_{\epsilon\rightarrow0}\iint\limits_{(t,x)\in T_{2}\bigcap K_{3}} |k_{1}\widetilde{x}^{-\frac{2}{3}}| \diff x
\diff t+\max|\phi_{x}|\lim\limits_{\epsilon\rightarrow0}\iint\limits_{(t,x)\in  T_{2}\bigcap K_{3}} |m_{7}\widetilde{x}^{-\frac{1}{3}}| \diff x\diff t\nonumber\\
& &+\max|\phi_{t}|\lim\limits_{\epsilon\rightarrow0}\iint\limits_{(t,x)\in  T_{2}\bigcap K_{2}} |k_{5}\widetilde{x}^{-\frac{2}{3}}| \diff x
\diff t+\max|\phi_{x}|\lim\limits_{\epsilon\rightarrow0}\iint\limits_{(t,x)\in  T_{2}\bigcap K_{2}} |m_{5}\widetilde{x}^{-\frac{1}{3}}| \diff x\diff t\nonumber\\
&\leq&\max|\phi_{t}|\lim\limits_{\epsilon\rightarrow0}(\int_{-\epsilon}^{-G_{1}(-\widetilde{\tau})^{\frac{3}{2}}}\int_{-\epsilon}^{0}+\int_{G_{1}(-\widetilde{\tau})^{\frac{3}{2}}}^{\epsilon}\int_{-\epsilon}^{0})(|k_{1}\widetilde{x}^{-\frac{2}{3}}|+
|k_{5}\widetilde{x}^{-\frac{2}{3}}|)d\widetilde{t}d\widetilde{x}\nonumber\\
& &+\max|\phi_{x}|\lim\limits_{\epsilon\rightarrow0}(\int_{-\epsilon}^{-G_{1}(-\widetilde{\tau})^{\frac{3}{2}}}\int_{-\epsilon}^{0}+\int_{G_{1}(-\widetilde{\tau})^{\frac{3}{2}}}^{\epsilon}\int_{-\epsilon}^{0})(
|m_{7}\widetilde{x}^{-\frac{1}{3}}|+|m_{5}\widetilde{x}^{-\frac{1}{3}}|)d\widetilde{t}d\widetilde{x}\nonumber\\
&=&0.\nonumber
\end{eqnarray}
Here
$$max|\phi_{x}|=sup\:\{|\phi_{x}(t,x)|\mid (t,x)\in D^{-}_{\epsilon}\},\quad max|\phi_{t}|=sup\:\{|\phi_{t}(t,x)|\mid (t,x)\in D^{-}_{\epsilon}\}$$ and $\tilde{\tau}$ is defined by (5.11) and (5.12).
So we have $|P|\rightarrow0$, as $\epsilon\rightarrow0$. That is to say, (5.8) holds.

Similarly, we can prove (5.9).
Thus, the theorem is proved.
\end{proof}
\begin{remark}
Here and throughout the following, we will use the convention $\bar{A}\approx \bar{B}$ whenever $\bar{C}^{-1}\bar{A}\leq \bar{B}\leq \bar{C}\bar{A}$ for a constant $\bar{C}\neq0$.
\end{remark}
\section{Delta-like solutions}
In this section, by the method of characteristics, we construct some weak solutions with a new kind of singularities, named ``Delta-like'' solution.\\
\textbf{6.1. Delta-like solution with point-shape singularity.}

We first consider a simple case, in which we assume

Assumption (A1):
$$
\Lambda_{-}(x)<\Lambda_{+}(x),\quad \forall \:x\in \mathbb{R};
\eqno(6.1)
$$
Furthermore, we assume that there exist $\alpha_{0}$ and $\beta_{0}$ with $\alpha_{0}<\beta_{0}$ satisfying

Assumption (A2):
$$
\Lambda_{-}(\alpha_{0})=\Lambda_{+}(\beta_{0})=0;
\eqno(6.2)
$$

Assumption (A3):$$
\Lambda^{'}_{-}(\alpha_{0})=\Lambda^{'}_{+}(\beta_{0})=0;
\eqno(6.3)
$$

Assumption (A4):$$
\Lambda^{''}_{-}(\alpha_{0})<0,\quad \Lambda^{''}_{+}(\beta_{0})>0.
\eqno(6.4)
$$
$\Lambda_{\pm}(x)$ satisfying assumptions (A1)-(A4) are shown in Figure 6.
\begin{figure}[h]
\centering
\includegraphics[trim=0 0 0 0, scale=0.80]{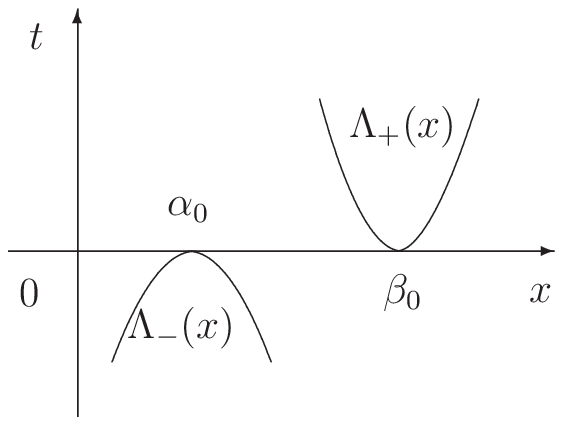}
\caption{$\Lambda_{\pm}(x)$ satisfying the assumptions (A1)-(A4)}
\end{figure}

We now derive the Delta-like solution with point-shape singularity
\begin{theorem}
Under the assumptions (A1)-(A4), it holds that in the neighborhood of the blowup point
\begin{displaymath}
|\rho| \leq \left\{\begin{array}{ll}
B_{1}|\widetilde{t}|^{-2},&\textrm{$\widetilde{x}=o(\widetilde{t}^{3})$},\\
B_{2}|\widetilde{x}|^{-\frac{2}{3}},&\textrm{$\widetilde{t}^{3}=o(\widetilde{x})$},\\
B_{3}|\widetilde{x}|^{-\frac{2}{3}},&\textrm{$\widetilde{x}=O(1)\widetilde{t}^{3}$}
\end{array}\right.
\end{displaymath}
and
\begin{displaymath}
|u| \leq \left\{\begin{array}{cc}
B_{4}|\widetilde{t}|^{2},&\textrm{$\widetilde{x}=o(\widetilde{t}^{3})$},\\
B_{5}|\widetilde{x}|^{\frac{2}{3}},&\textrm{$\widetilde{t}^{3}=o(\widetilde{x})$},\\
B_{6}|\widetilde{x}|^{\frac{2}{3}},&\textrm{$\widetilde{x}=O(1)\widetilde{t}^{3}$}
\end{array}\right.
\end{displaymath}
for sufficiently small $\widetilde{t}$ and $\widetilde{x}$ defined in Section 4 and $B_{i}\,(i=1,\cdots,6)$ are positive constants depending only on the initial data at $(\alpha_{0},\beta_{0})$.
Furthermore, the solution $(\rho,\,u)$ is Delta-like solution, which we call it Delta-like solution with point-shape singularity.
\end{theorem}
Before proving Theorem 6.1, we need the following lemmas
\begin{lemma}
Under the assumptions (A1)-(A4), there is only one singular point, i.e., $(t_{0},x_{0})$ defined by (3.12) and (3.13), on which $\rho=\infty$ and away from $(t_{0},x_{0})$, $\rho(t,x)$ is finite.
\end{lemma}
\begin{proof} By (3.12), (3.13), (4.21) and (6.2), we observe that at $(t_{0},x_{0})$, $\rho=\infty$.

Suppose that there exists another point $(t,x)\neq(t_{0},x_{0})$ such that $\rho(t,x)=\infty$. from $(t,x)$, there exist only two characteristics intersecting the $x$-axis at $\alpha$ and $\beta$ respectively. By (4.21) we have
$$
\Lambda_{-}(\alpha)=\Lambda_{+}(\beta).
\eqno(6.15)
 $$
If $\alpha\neq\alpha_{0}$, then by the assumptions (A1)-(A4) we have

$$\Lambda_{-}(\alpha)<0,\quad \Lambda_{+}(\beta)\geq0,\quad \forall \:\beta\in \mathbb{R}.$$
This contradicts to (6.15).

Similarly, it is easy to show that the assumption $\beta\neq\beta_{0}$ also leads to a contradiction.
Thus, the lemma is proved.
\end{proof}
Under the assumptions (A1)-(A4), the characteristics can be depicted as follows: the characteristics $x^{-}$ and $x^{+}$ passing through $(0,\alpha_{0})$ and $(0,\beta_{0})$ respectively tangent at $(t_{0},x_{0})$ and then they turn away from each other (see Figure 7).
\begin{figure}[h]
\centering
\includegraphics[trim=0 0 0 0, scale=0.80]{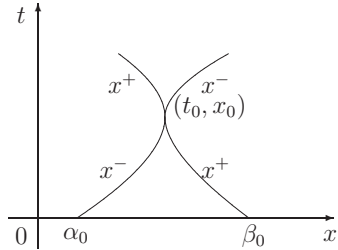}
\caption{The characteristics under assumptions (A1)-(A4).}
\end{figure}

By the same method as Lemmas 4.1-4.3, we get the following two lemmas without proof.
\begin{lemma}
Under the assumptions (A1)-(A4), in the neighborhood of $(t_{0},x_{0})$, it holds that
\begin{eqnarray}
\widetilde{\beta}&=&\left(\frac{1}{2}(-B_{7}\widetilde{x}-B_{8}\widetilde{t}^{3})+\sqrt{\frac{1}{4}(-B_{7}\widetilde{x}-B_{8}\widetilde{t}^{3})^{2}+\frac{1}{27}(B_{9}\widetilde{t}^{2})^{3}}\right)^{\frac{1}{3}}\nonumber\\
& &+\left(\frac{1}{2}(-B_{7}\widetilde{x}-B_{8}\widetilde{t}^{3})-\sqrt{\frac{1}{4}(-B_{7}\widetilde{x}-B_{8}\widetilde{t}^{3})^{2}+\frac{1}{27}(B_{9}\widetilde{t}^{2})^{3}}\right)^{\frac{1}{3}}
-\frac{B_{10}\widetilde{t}}{3},\nonumber
\end{eqnarray}
where $B_{i}$ $(i=7,\cdots,10)$ are constants depending only on the initial data at $(\alpha_{0},\beta_{0})$.
\end{lemma}
\begin{lemma}
Under the assumptions (A1)-(A4), in the neighborhood of $(t_{0},x_{0})$, it holds that
\begin{displaymath}
\widetilde{\beta} = \left\{\begin{array}{ll}
B_{11}\widetilde{t},&\textrm{$\widetilde{x}=o(\widetilde{t}^{3})$},\\
B_{12}\widetilde{x}^{\frac{1}{3}},&\textrm{$\widetilde{t}^{3}=o(\widetilde{x})$},\\
B_{13}\widetilde{x}^{\frac{1}{3}},&\textrm{$\widetilde{x}=O(1)\widetilde{t}^{3}$},
\end{array}\right.
\end{displaymath}
\begin{displaymath}
\Lambda_{+}(\beta)-\Lambda_{-}(\alpha) = \left\{\begin{array}{ll}
B_{14}\widetilde{t}^{2},&\textrm{$\widetilde{x}=o(\widetilde{t}^{3})$},\\
B_{15}\widetilde{x}^{\frac{2}{3}},&\textrm{$\widetilde{t}^{3}=o(\widetilde{x})$},\\
B_{16}\widetilde{x}^{\frac{2}{3}},&\textrm{$\widetilde{x}=O(1)\widetilde{t}^{3}$}
\end{array}\right.
\end{displaymath}
and
\begin{displaymath}
\Lambda_{+}(\beta)+\Lambda_{-}(\alpha) = \left\{\begin{array}{ll}
B_{17}\widetilde{t}^{2},&\textrm{$\widetilde{x}=o(\widetilde{t}^{3})$},\\
B_{18}\widetilde{x}^{\frac{2}{3}},&\textrm{$\widetilde{t}^{3}=o(\widetilde{x})$},\\
B_{19}\widetilde{x}^{\frac{2}{3}},&\textrm{$\widetilde{x}=O(1)\widetilde{t}^{3}$}
\end{array}\right.
\end{displaymath}
where $B_{i}$ $(i=11,\cdots,19)$ are constants depending only on the initial data at $(\alpha_{0},\beta_{0})$.
\end{lemma}
\textbf{The proof of Theorem 6.1.}
 The behavior of $\rho$ and $u$ can be derived easily from the above two lemmas. By the same method as Theorem 5.1, in the integral domain $D_{\epsilon}$ defined below, we have $|\rho|\leq |\widetilde{x}|^{-\frac{2}{3}}$ and the orders of $|\rho u|$ and $|\rho u^{2}|$ are higher than the order of $|\rho|$, thus, (5.3) holds obviously. Next, we prove that system (2.1) satisfies the Definition 5.2. First we consider the first equation of (2.1).

 Define
 $$
 D_{\epsilon}=\{(t,x)|\:|t-t_{0}|\leq\epsilon,\;|x-x_{0}|\leq\epsilon\}.
 $$

 For arbitrary $\phi(t,x)\in C_{0}^{1}$, it holds that
  \begin{eqnarray}
& &\iint\limits_{t\geq0} (\rho\phi_{t}+\rho u\phi_{x}) dxdt+\int_{t=0}\rho_{0}\phi dx\nonumber\\
&=&\iint\limits_{\{t\geq0\}-D_{\epsilon}} (\rho\phi_{t}+\rho u\phi_{x}) dxdt+\int_{t=0}\rho_{0}\phi+
\iint\limits_{D_{\epsilon}} (\rho\phi_{t}+\rho u\phi_{x}) dxdt\nonumber\\
&=&\left(\int_{0}^{\infty}\int_{-\infty}^{x_{0}-\epsilon}+\int_{0}^{\infty}\int_{x_{0}+\epsilon}^{+\infty}
+\int_{0}^{t_{0}-\epsilon}\int_{x_{0}-\epsilon}^{x_{0}+\epsilon}+
\int_{t_{0}+\epsilon}^{\infty}\int_{x_{0}-\epsilon}^{x_{0}+\epsilon}\right)(\rho\phi_{t}+\rho u\phi_{x}) dxdt\nonumber\\
& &
+
\iint\limits_{D_{\epsilon}} (\rho\phi_{t}+\rho u\phi_{x}) dxdt+\int_{t=0}\rho_{0}\phi dx\nonumber\\
&=&\left(\int_{0}^{\infty}\int_{-\infty}^{x_{0}-\epsilon}+\int_{0}^{\infty}\int_{x_{0}+\epsilon}^{+\infty}
+\int_{0}^{t_{0}-\epsilon}\int_{x_{0}-\epsilon}^{x_{0}+\epsilon}+
\int_{t_{0}+\epsilon}^{\infty}\int_{x_{0}-\epsilon}^{x_{0}+\epsilon}\right)(\rho_{t}\phi+(\rho u)_{x}\phi) dxdt \nonumber\\
& &-\int_{-\infty}^{+\infty}\rho_{0}\phi dx+\int_{x_{0}-\epsilon}^{x_{0}+\epsilon}(\rho(t_{0}-\epsilon)\phi(t_{0}-\epsilon)-\rho(t_{0}+\epsilon)\phi(t_{0}+\epsilon))dx
+
\iint\limits_{D_{\epsilon}} (\rho\phi_{t}+\rho u\phi_{x}) dxdt\nonumber\\& &+\int_{t_{0}-\epsilon}^{t_{0}+\epsilon}(\rho u\phi(x_{0}-\epsilon)-\rho u\phi(x_{0}+\epsilon))dt+\int_{-\infty}^{+\infty}\rho_{0}\phi dx
\nonumber\\
&=&\left(\int_{0}^{\infty}\int_{-\infty}^{x_{0}-\epsilon}+\int_{0}^{\infty}\int_{x_{0}+\epsilon}^{+\infty}
+\int_{0}^{t_{0}-\epsilon}\int_{x_{0}-\epsilon}^{x_{0}+\epsilon}+
\int_{t_{0}+\epsilon}^{\infty}\int_{x_{0}-\epsilon}^{x_{0}+\epsilon}\right)(\rho_{t}+(\rho u)_{x})\phi dxdt \nonumber\\
& &+\int_{x_{0}-\epsilon}^{x_{0}+\epsilon}(\rho\phi(t_{0}-\epsilon)-\rho\phi(t_{0}+\epsilon))dx+\int_{t_{0}-\epsilon}^{t_{0}+\epsilon}(\rho u\phi(t_{0}-\epsilon)-\rho u\phi(t_{0}+\epsilon))dt\nonumber\\
& &+
\iint\limits_{D_{\epsilon}} (\rho\phi_{t}+\rho u\phi_{x}) dxdt\nonumber\\
&\triangleq&M_{1}+M_{2}+M_{3}+M_{4}.\nonumber
\end{eqnarray}
In the limit $\epsilon\rightarrow0$, $M_{1}\rightarrow0$ due to the first equation of (2.1), $M_{4}\rightarrow0$ due to (5.3), for $M_{2}$, we have\\
$$
|M_{2}|\leq2max|\phi|\int_{x_{0}-\epsilon}^{x_{0}+\epsilon}|\rho|dx\leq2max|\phi|\int_{x_{0}-\epsilon}^{x_{0}+\epsilon}|x-x_{0}|^{-\frac{2}{3}}dx\leq12max|\phi|\epsilon^{\frac{1}{3}},
$$
where $$max|\phi|=sup\:\{|\phi(t,x)|\mid (t,x)\in D_{\epsilon}\}.$$
So $M_{2}$ tends to zero. For $M_{3}$, since $|\rho u \phi|$ is finite on $(t_{0}-\epsilon,t_{0}+\epsilon)$. Thus, $M_{3}\rightarrow0$ as $\epsilon\rightarrow0$.
Similarly, for the second equation of (2.1) we have
$$
\iint\limits_{t\geq0} (\rho u\phi_{t}+(\rho u^{2}+p)\phi_{x}) dxdt+\int_{t=0}\rho_{0}u_{0}\phi dx=0.
$$
\textbf{6.2. Delta-like solution with line-shape singularity: Type I.}

In this subsection, we are going to investigate another important Delta-like solution named Delta-like solution with line-shape singularity. To do so, we assume that

Assumption ($B1$):
$$
\Lambda_{-}(x)<\Lambda_{+}(x), \quad \forall\: x\in \mathbb{R};
\eqno(6.5)
$$
Furthermore, there exist $\alpha_{0}$ and $\beta_{0}$ with $\alpha_{0}<\beta_{0}$ satisfying

Assumption ($B2$):
 $$
\Lambda_{-}(\alpha_{0})=\Lambda_{+}(\beta)=0,\quad \forall \beta_{0}\leq\beta\leq\widehat{\beta};
\eqno(6.6)
$$
$$(\text{resp.}
\quad\Lambda_{-}(\alpha)=\Lambda_{+}(\beta_{0})=0,\quad \forall\: \widehat{\alpha}\leq\alpha\:\leq\alpha_{0};)$$

Assumption ($B3$): $$
\Lambda^{'}_{-}(\alpha_{0})=\Lambda^{'}_{+}(\beta_{0})=0;
\eqno(6.7)
$$

Assumption ($B4$): $$
\Lambda^{''}_{-}(\alpha_{0})<0,\quad \Lambda^{''}_{+}(\beta_{0})=0;
\eqno(6.8)
$$
$$(\text{resp}.\quad\Lambda^{''}_{-}(\alpha_{0})=0,\quad \Lambda^{''}_{+}(\beta_{0})>0.)$$
$\Lambda_{\pm}(x)$ satisfying assumptions ($B1$)-($B4$) are shown in Figure 8.
\begin{figure}[h]
\centering
\includegraphics[trim=0 0 0 0,scale=0.80]{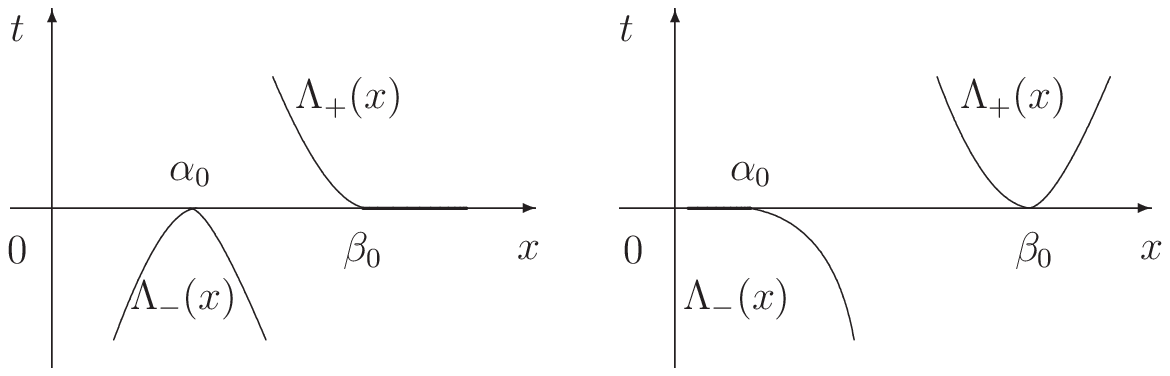}
\caption{$\Lambda_{\pm}(x)$ satisfying the assumptions $(B1)$-$(B4)$.}
\end{figure}

Define
$$
\widehat{t}=\int_{\alpha_{0}}^{\widehat{\beta}} \frac{1}{\Lambda_{+}(\zeta)-\Lambda_{-}(\zeta)}d\zeta.
$$
Under the assumptions $(B1)$-$(B4)$, we get the Delta-like solution with line-shape singularity
\begin{theorem}
Under the assumptions ($B1$)-($B4$), it holds that
$$
\lim\limits_{\epsilon\rightarrow0}\int_{t_{0}-\epsilon}^{\widehat{t}+\epsilon}\int_{x_{0}-\epsilon}^{x_{0}+\epsilon}(\rho\phi_{t}
+\rho u\phi_{x})dxdt=0
\eqno(6.9)
$$
and
$$
\lim\limits_{\epsilon\rightarrow0}\int_{t_{0}-\epsilon}^{\widehat{t}+\epsilon}\int_{x_{0}-\epsilon}^{x_{0}+\epsilon}(\rho u\phi_{t}
+(\rho u^{2}+p)\phi_{x})dxdt=0.
\eqno(6.10)
$$
Furthermore, the solution derived by the method of characteristics is a Delta-like solution and we call it Delta-like solution with line-shape singularity.
\end{theorem}
Before proving Theorem 6.2, we need the following lemmas.
\begin{lemma}
Under the assumptions ($B1$)-($B4$), the singularities form a line $L\triangleq\{(t,x)\mid x=x_{0},t_{0}\leq t\leq \widehat{t} \}$, and on this line, it holds that $\rho=\infty$, while off the line, $\rho$ is finite and smooth, we denote this set by $L^{c}$.
\end{lemma}
\begin{proof} By (4.21), $\rho=\infty$ if and only if $\Lambda_{+}(\beta)=\Lambda_{-}(\alpha),$ on the other hand, by (2.10)-(2.11), (3.12)-(3.13) and the assumptions ($B1$)-($B4$), when $\beta_{0}\leq\beta\leq\widehat{\beta}$,
\begin{eqnarray}
x&=&\frac{1}{2}\left\{\alpha_{0}+\beta+\int_{\alpha_{0}}^{\beta}\frac{\Lambda_{+}(\zeta)+\Lambda_{-}(\zeta)}{\Lambda_{+}(\zeta)-\Lambda_{-}(\zeta)}d\zeta\right\}\nonumber\\
&=&\frac{1}{2}\left\{\alpha_{0}+\beta_{0}+\beta-\beta_{0}+\int_{\alpha_{0}}^{\beta_{0}}\frac{\Lambda_{+}(\zeta)+\Lambda_{-}(\zeta)}{\Lambda_{+}(\zeta)-\Lambda_{-}(\zeta)}d\zeta
+\int_{\beta_{0}}^{\beta}\frac{\Lambda_{+}(\zeta)+\Lambda_{-}(\zeta)}{\Lambda_{+}(\zeta)-\Lambda_{-}(\zeta)}d\zeta\right\}\nonumber\\
&=&\frac{1}{2}\left\{2x_{0}+\beta-\beta_{0}+\int_{\beta_{0}}^{\beta}(-1)d\zeta\right\}=x_{0}\nonumber
\end{eqnarray}
and
\begin{eqnarray}
t&=&\int_{\alpha_{0}}^{\beta}\frac{1}{\Lambda_{+}(\zeta)-\Lambda_{-}(\zeta)}d\zeta\nonumber\\
&=&
\int_{\alpha_{0}}^{\beta_{0}}\frac{1}{\Lambda_{+}(\zeta)-\Lambda_{-}(\zeta)}d\zeta+
\int_{\beta_{0}}^{\beta}\frac{1}{\Lambda_{+}(\zeta)-\Lambda_{-}(\zeta)}d\zeta\nonumber\\
&=&t_{0}+
\int_{\beta_{0}}^{\beta}\frac{1}{\Lambda_{+}(\zeta)-\Lambda_{-}(\zeta)}d\zeta\geq t_{0}.\nonumber
\end{eqnarray}
Moreover, since $\beta\leq\widehat{\beta}$, we have $t\leq\widehat{t}$.
So, the mapping $\Pi$ maps the curve $\Lambda_{-}(\alpha)=\Lambda_{+}(\beta)$ into $L$ and $(t_{0},x_{0})$ is the blowup point. Passing through any point $(t,x)\in L^{c}$, there exist only two characteristics which intersect the $x$-axis at $\alpha$ and $\beta$ with $\alpha<\beta$ and satisfying  $$\Lambda_{-}(\alpha)\neq\Lambda_{+}(\beta),$$ i.e., $\rho<\infty.$ Thus, the proof of Lemma 6.4 is completed.
\end{proof}
\begin{remark}
The characteristics under the assumptions ($B1$)-($B4$) can be depicted as follows: the characteristics $x^{-}$ and $x^{+}$ passing through $(0,\alpha_{0})$ and $(0,\beta_{0})$, respectively, are tangent at $(t_{0},x_{0})$, and then they turn away from each other when $t>t_{0}$. $L$ is the envelope of the characteristics passing through the points $(0,\alpha_{0})$ and $(0,\beta)$ in which $\beta_{0}\leq\beta\leq\widehat{\beta}$ (see Figure 9).
\end{remark}
\begin{figure}[h]
\centering
\includegraphics[trim= 0 0 0 0, scale=0.80]{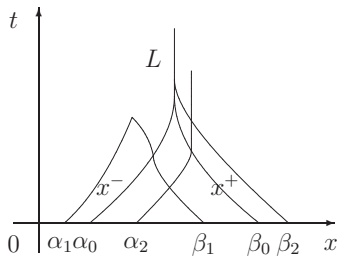}
\caption{A sketch of the characteristics under the assumptions ($B1$)-($B4$).}
\end{figure}

\begin{lemma}
Under the assumptions ($B1$)-($B4$), in the neighborhood of $(t,x_{0})\in L$ defined in Lemma 6.4, it holds that
$$
\widetilde{\alpha}=V_{1}\widetilde{x}^{\frac{1}{3}},
$$
$$
\Lambda_{+}(\beta)-\Lambda_{-}(\alpha)=V_{2}\widetilde{x}^{\frac{2}{3}},
$$
$$
\Lambda_{+}(\beta)+\Lambda_{-}(\alpha)=-V_{2}\widetilde{x}^{\frac{2}{3}}.
$$
Where $V_{i}\,(i=1,2)$ are constants depending only on the initial data at $\alpha_{0}$.
\end{lemma}
\begin{proof}By the assumptions ($B1$)-($B4$) and (4.2), we have
$$
\widetilde{x}=-\frac{1}{6}\left(\frac{2\Lambda_{-}^{''}(\alpha_{0})-\Lambda_{+}^{''}(\alpha_{0})}{\Lambda_{+}(\alpha_{0})}\right)\widetilde{\alpha}^{3}.
$$
So
 $$
 \widetilde{\alpha}=\sqrt[3]{-6\frac{\Lambda_{+}(\alpha_{0})}{2\Lambda_{-}^{''}(\alpha_{0})-\Lambda_{+}^{''}(\alpha_{0})}}\widetilde{x}^{\frac{1}{3}}\triangleq V_{1}\widetilde{x}^{\frac{1}{3}}.
$$
On the other hand
$$
\Lambda_{+}(\beta)-\Lambda_{-}(\alpha)=-\frac{\Lambda_{-}^{''}(\alpha_{0})}{2}\widetilde{\alpha}^{2}\triangleq V_{2}\widetilde{x}^{\frac{2}{3}}
$$
and
$$
\Lambda_{+}(\beta)+\Lambda_{-}(\alpha)=\frac{\Lambda_{-}^{''}(\alpha_{0})}{2}\widetilde{\alpha}^{2}\triangleq -V_{2}\widetilde{x}^{\frac{2}{3}}.
$$
Thus, the lemma is proved.
\end{proof}

By the above two lemmas, we are ready to prove Theorem 6.2\\
\textbf{The proof of Theorem 6.2}.

By Lemma 6.5,
$$|\rho|\leq 2V_{2}|\widetilde{x}|^{-\frac{2}{3}},\:|\rho u|\leq 2\mu\:\text{ and}\:
|\rho u^{2}|\leq 2\mu V_{2}|\widetilde{x}|^{\frac{2}{3}},\;\text{in}\:\widehat{D}(\epsilon),$$
where
$$
\widehat{D}(\epsilon)=\{(t,x)|\,t_{0}-\epsilon\leq t\leq \widehat{t}+\epsilon,\;|x-x_{0}|\leq\epsilon\}.
$$
Obviously, the singular integral (5.3) in $\widehat{D}(\epsilon)$ is convergent.
By the same method as Theorem 6.1, for arbitrary $\phi\in C_{0}^{1}(\mathbb{R^{+}}\times\mathbb{R})$, we have
\begin{eqnarray}
& &\iint\limits_{t\geq0} (\rho\phi_{t}+\rho u\phi_{x}) dxdt+\int_{t=0}\rho_{0}\phi dx\nonumber\\
&=&\int_{t_{0}-\epsilon}^{\widehat{t}+\epsilon}\int_{x_{0}-\epsilon}^{x_{0}+\epsilon}(\rho\phi_{t}+\rho u\phi_{x})dx dt+\int_{t_{0}-\epsilon}^{\widehat{t}+\epsilon}(\rho u\phi(x_{0}+\epsilon)-\rho u\phi(x_{0}-\epsilon))dt\nonumber\\
&\triangleq&N_{1}+N_{2}.\nonumber
\end{eqnarray}
When $\epsilon\rightarrow0$, by (5.3) we have $$N_{1}\rightarrow0.$$  For $N_{2}$,
since $|\rho u|\leq 2\mu$, $|\widehat{t}-t_{0}|<\infty$ and
$$\lim\limits_{\epsilon\rightarrow0}(\rho u(x_{0}-\epsilon)-\rho u(x_{0}+\epsilon))=(-\mu-(-\mu))=0,$$
by the dominated convergence theorem,
 $$N_{2}\rightarrow0.$$

Similarly, when $\epsilon\rightarrow0$, we have
\begin{eqnarray}
& &\iint\limits_{t\geq0} (\rho u\phi_{t}+(\rho u^{2}+p)\phi_{x}) dxdt+\int_{t=0}\rho_{0}u_{0}\phi dx\rightarrow0\nonumber
\end{eqnarray}
Thus, the theorem is proved.

\textbf{6.3. Delta-like solution with line-shape singularity: Type II.}

This section is a continuation of the previous subsection 6.2. The difference is the assumptions on initial data. Here we assume

Assumption ($C1$):
$$
\Lambda_{-}(x)<\Lambda_{+}(x),\quad \forall\: x\in \mathbb{R};
\eqno(6.11)
$$
Furthermore, we suppose that there exist $\alpha_{0}$,\:$\beta_{0}$ and $\alpha_{0}<\beta_{0}$ satisfying

Assumption ($C2$):
 $$
\Lambda_{-}(\alpha)=\Lambda_{+}(\beta)=0,\quad \forall\, \widehat{\alpha}\leq\alpha\:\leq\alpha_{0},\quad \forall\:\beta_{0}\leq\beta\leq\widehat{\beta};
\eqno(6.12)
$$

Assumption ($C3$): $$
\Lambda^{'}_{-}(\alpha_{0})=\Lambda^{'}_{+}(\beta_{0})=0;
\eqno(6.13)
$$

Assumption ($C4$): $$
\Lambda^{''}_{-}(\alpha_{0})=0,\quad \Lambda^{''}_{+}(\beta_{0})=0.
\eqno(6.14)
$$
$\Lambda_{\pm}(x)$ satisfying assumptions ($C1$)-($C4$) can be shown in Figure 10.
\begin{figure}[h]
\centering
\includegraphics[trim=0 0 0 0, scale=0.80]{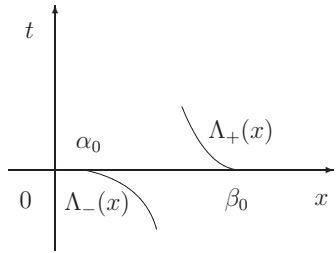}
\caption{A sketch of $\Lambda_{\pm}(x)$ satisfying the assumptions ($C1$)-($C4$).}
\end{figure}

By the same method as Lemma 6.4, we have the following lemma.
\begin{lemma}
Under the assumptions ($C1$)-($C4$), the singularities form a line $L$ which is defined as Lemma 6.4. The density $\rho$ is infinite on the line $L$, while is finite and smooth off the line.
\end{lemma}
\begin{remark}
Under the assumptions ($C1$)-($C4$), the characteristics can be shown in Figure 11. The difference between the assumptions ($B1$)-($B4$) and the assumptions ($C1$)-($C4$) is that once the characteristic touch at the line $L$, they will remain in contact.
\end{remark}
\begin{figure}[h]
\centering
\includegraphics[trim=0 0 0 0, scale=0.80]{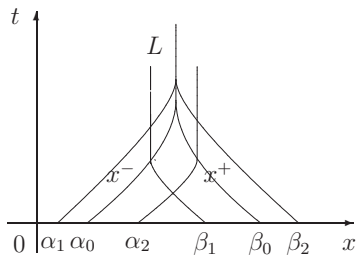}
\caption{A sketch of the characteristics under the assumptions ($C1$)-($C4$).}
\end{figure}

\begin{remark}
On the line $L$ the solution constructed by the method of characteristics satisfies $u^{+}=u^{-}=0$, i.e., it satisfies (5.5) under the assumptions ($B1$)-($B4$) and ($C1$)-($C4$).
\end{remark}
\section{conclusion}
In this paper, we study the behavior of one-dimensional Chaplygin gas. In particular, we analyze the formation of singularities for such a system.
We show that these singularities are very different from the traditional formation of singularities, such as in the case of shock wave formation. We call these new type singularities ``Delta-like'' singularities since the densities become infinite at the singularities. Depending on the initial conditions, different types of Delta-like singularities can form, such as Delta-like solution with point-shape singularity and Delta-like solution with line-shape singularity (including Type I and II). For convenience, we assume that the initial data only leads to the formation of one Delta-like singularity. It is straight forward to generalize this to the cases which allow many Delta-like singularities to form. More specially, we can generalize the assumptions (A2)-(A4) to the case:
there exist numerous $\alpha_{i}$, $\beta_{j}$ $(i,j=0,1,2,\cdots,n)$ satisfying

Assumptions (A2'):
$$
\Lambda_{-}(\alpha_{i})=\Lambda_{+}(\beta_{j})=0;
\eqno(7.1)
$$

Assumptions (A3')
$$
\Lambda^{'}_{-}(\alpha_{i})=\Lambda^{'}_{+}(\beta_{j})=0;
\eqno(7.2)
$$

Assumptions (A4')
$$
\Lambda^{''}_{-}(\alpha_{i})<0,\quad \Lambda^{''}_{+}(\beta_{j})>0.
\eqno(7.3)
$$
Under the assumptions (A1), (A2')-(A4'), there are numerous Delta-like solution with point-shape singularity and the theory of subsection 6.1 holds in this case.

We have studied the formation of singularities under specific assumptions on initial data in this paper. We believe that other initial data could lead to more complicated and more interesting phenomena of singularity formation. However, the study of such possibility is beyond the scope of this paper.

For the more interesting and complicated cusp-type singularity, we will construct the Delta-like solution in the forthcoming paper.
\vskip 4mm
\noindent{\Large {\bf Acknowledgements.}} Kong and Wei thank A. Bressan for helpful discussion.
The work of Kong and Wei was supported in
part by the NNSF of China (Grant No.: 11271323), Zhejiang Provincial Natural Science Foundation of China (Grant No.: Z13A010002) and a National Science and Technology Project during the twelfth five-year plan of China (2012BAI10B04).
The work of Zhang was supported by the Research Grant Council of HKSAR (Grant No.: CityU 103509).


\begin{thebibliography}{999}
\bibitem{A} S. Alinhac, Blowup for nonlinear hyperbolic equations, Progress in Nonlinear Differential Equations and Their Applications 17, Birkh$\ddot{a}$user, (1995).\vspace{2mm}
\bibitem{1} Y. Brenier, Some geometric PDEs related to hydrodynamics and electrodynamics, Proceedings of ICM, \textbf{3} (2002), 761-772.\vspace{2mm}
\bibitem{Ch1} D. Christodoulou, The formation of shocks in 3-dimensional fluids, EMS Monographs in Mathematics, Z$\ddot{u}$rich, (2007).\vspace{2mm}
\bibitem{Ch2} D. Christodoulou and S. Miao, Compressible flow and Euler's equations, arXiv:1212.1867 v1.\vspace{2mm}
\bibitem{2} C. M. Dafermos, Generalized characteristics and the structure of solutions of hyperbolic conservation laws, Indiana University Mathematics Journal, \textbf{26} (1977), 1097-1119.\vspace{2mm}
\bibitem{3} J. Eggers and J. Hoppe, Singularity for time-like extremal hypersurfaces, Physics Letter B, \textbf{680} (2009), 274-278.\vspace{2mm}
\bibitem{4} M. Golubitsky, An introduction to catastrophy theorey and its aplications, SIAM Review, \textbf{20} (1987), 352-387.\vspace{2mm}
\bibitem{5} L. H\"{o}rmander, Lectures on nonlinear hyperbolic differential equations, Math\'{e}matiques and Applications \textbf{26}, Springer, (1997).\vspace{2mm}
\bibitem{6} F. John, Formation of singularities in one-dimensional nonlinear wave propagation, Comm. Pure Appl. Math., \textbf{27} (1974), 377-405.\vspace{2mm}
\bibitem{7} D.-X. Kong, Cauchy problem for quasilinear hyperbolic systems, MSJ Memoiros 6, the Mathimatical Society of Japan, Tokyo, (2000).\vspace{2mm}
\bibitem{8} D.-X. Kong, Life-span of classical solutions to quasilinear hyperbolic systems with slow decay initial data, Chin. Ann. Math., \textbf{21B} (2000), 413-440.\vspace{2mm}
\bibitem{9} D.-X. Kong, Formation and propagation of singularities for $2\times2$ quasilinear hyperbolic systems,  Transactions of American Mathematical Society, \textbf{354} (2002), 3155-3179.\vspace{2mm}
\bibitem{10} D.-X. Kong and M. Tusji, Global solutions for $2\times2$ hyperbolic systems with linearly degenerate characteristics, Funkcialaj Ekvacioj, \textbf{42} (1999), 129-155.\vspace{2mm}
\bibitem{11} D.-X. Kong and Qiang Zhang, Solutions formula and time-periodicity for the motion of relativistic strings in the minkowski space $R^{1+n}$, Physica D, \textbf{238} (2009), 902-922.\vspace{2mm}
\bibitem{12} M. P. Lebaud, Description de le formation d'un choc dans le \textit{p}-syst\`{e}me, J. Math. Pures Appl., \textbf{73} (1994), 523-565.\vspace{2mm}
\bibitem{15} P. D. Lax, Hyperbolic systems of conservation laws II, Comm. Pure Appl. Math., \textbf{10} (1957), 537-566.\vspace{2mm}
\bibitem{13} S. Nakane, Formation of shocks for a single conservation law, SIAM J. Math. Anal., \textbf{19} (1988), 1391-1408.\vspace{2mm}
\bibitem{14} J. Smoller, Shock waves and reaction-diffusion equations, Springer-Verlag, (1994).





    \end{thebibliography}
\end{document}